\documentclass[11pt]
{amsart}
\usepackage{amssymb,amsmath,amsthm,amsfonts,amsopn,url,color,hyperref,enumerate,mathtools,microtype,MnSymbol}
\usepackage[normalem]{ulem}
\usepackage[all]{xy}
\usepackage{amscd}

\theoremstyle{plain}
\newtheorem{thm}{Theorem}[section]

\newtheorem{prop}[thm]{Proposition}
\newtheorem{lemma}[thm]{Lemma}
\newtheorem{cor}[thm]{Corollary}

\newtheorem{defpr}[thm]{Definition/Proposition}
\theoremstyle{definition}
\newtheorem{defn}[thm]{Definition}
\newtheorem*{defn*}{Definition}
\newtheorem*{question*}{Question}

\newtheorem{example}[thm]{Example}
\newtheorem*{example*}{Example}
\newtheorem{rem}[thm]{Remark}
\newtheorem*{rem*}{Remark}
\newtheorem{construction}[thm]{Construction}
\newcommand{\field}[1]{\mathbb{#1}}
\newcommand{\N}{\field{N}}
\newcommand{\Z}{\field{Z}}

\newcommand{\ideal}[1]{\mathfrak{#1}}
\newcommand{\m}{\ideal{m}}

\newcommand{\p}{\ideal{p}}

\newcommand{\func}[1]{\mathrm{#1} \,}

\newcommand{\Spec}{\func{Spec}}

\newcommand{\im}{\func{im}}

\newcommand{\arrow}[1]{\stackrel{#1}{\rightarrow}}

\newcommand{\ra}{\rightarrow}

\DeclareMathOperator{\ann}{ann}

\DeclareMathOperator{\Hom}{Hom}

\newcommand{\be}{\begin{enumerate}}
\newcommand{\ee}{\end{enumerate}}

\newcommand{\li}
 {\leftfootline}

\newcommand{\onto}{\twoheadrightarrow}

\newcommand{\into}{\hookrightarrow}

\newcommand{\cA}{\mathcal{A}}

\newcommand{\cM}{\mathcal{M}}
\newcommand{\cP}{\mathcal{P}}
\renewcommand{\phi}{\varphi}
\DeclareMathOperator{\Frac}{Frac}

\DeclareMathOperator{\Soc}{Soc}

\newcommand{\subsel}{submodule selector}
\newcommand{\extop}{extensive operation}
\newcommand{\funcss}{functorial}
\newcommand{\funceo}{functorial}
\newcommand{\surfuncss}{surjection-functorial}
\newcommand{\surfunceo}{surjection-functorial}
\newcommand{\radss}{co-idempotent}
\newcommand{\reseo}{residual}
\newcommand{\resop}{residual operation}
\newcommand{\opsub}{order-preserving on submodules}
\newcommand{\opamb}{order-preserving on ambient modules}
\newcommand{\idemeo}{idempotent}
\newcommand{\idemss}{idempotent}
\newcommand{\dual}{\smallsmile}
\newcommand{\cl}{{\mathrm{cl}}}
\let\int\relax
\DeclareMathOperator{\int}{i}

\newcommand{\nzd}{non-zerodivisor}
\newcommand{\fg}{finitely generated}
\newcommand{\charp}{characteristic $p>0$}
\DeclareMathOperator{\tto}{to}

\DeclareMathOperator{\tr}{tr}
\DeclareMathOperator{\dv}{div}
\DeclareMathOperator{\tom}{tom}
\newcommand{\CM}{Cohen-Macaulay}
\newcommand{\sop}{system of parameters}

\author{Neil Epstein}
\address{Department of Mathematical Sciences \\ George Mason University \\ Fairfax, VA  22030}
\email{nepstei2@gmu.edu}

\author{Rebecca R.G.}
\address{Department of Mathematical Sciences \\ George Mason University \\ Fairfax, VA  22030}
\email{rrebhuhn@gmu.edu}

\title{Closure-interior duality over complete local rings}

\date{April 22, 2021}
\begin{document}
\begin{abstract}
We define a duality operation connecting closure operations, interior operations, and test ideals, and describe how the duality acts on common constructions such as trace, torsion, tight and integral closures, and divisible submodules. This generalizes the relationship between tight closure and tight interior given in \cite{nmeSc-tint} and allows us to extend commonly used results on tight closure test ideals to operations such as those above.
\end{abstract}

\maketitle
\setcounter{tocdepth}{1} 
\tableofcontents

\section{Introduction}\label{sec:intro}
Closure operations have long been an important subject in commutative algebra (c.f. the first named author's survey \cite{nme-guide2}).  Many of these closures, including tight closure, are  \emph{residual}, i.e., the closure of a submodule $N$ of an $R$-module $M$ is computable from the closure of $0$ in $M/N$.  Coordinated with this is the study of \emph{test ideals}, traditionally a part of the study of tight closure theory.  In its simplest form, the test ideal is the ideal of ring elements that uniformly kill the closure of zero in every module.  In the special case of complete local rings, the tight closure test ideal can be recovered from the  Matlis dual to the closure of zero in the injective hull of the residue field.  Duals to closures of zero in other modules can then be seen as yielding an \emph{interior} operation, as detailed in \cite{nmeSc-tint}. The specific case of module closures and trace ideals is discussed in \cite{PeRG}.

Techniques for analyzing these structures are scattered throughout the literature. In this paper, we prove that this duality holds for all residual closure operations on complete local rings, demonstrate how properties of closures/interiors pass to their dual operations, and show how this framework applies to closures and interiors that appear throughout the literature.

The paper is structured as follows.

Section~\ref{sec:first} introduces our most general context of \subsel s and \extop s on a category of modules over an associative ring. In it, we give a bijection (see Proposition~\ref{pr:resid}) between \resop s on pairs of modules $N \subseteq M$ and \subsel s on (single) modules.  We then detail how properties of \subsel s correspond to properties of \resop s across this bijection (see Proposition~\ref{pr:ssres}).  This allows us to reduce every aspect of the study of \reseo\ closure operations to the study of a certain class of submodule selectors -- namely the \radss\ preradicals.

In Section~\ref{sec:ssdual}, we restrict our attention to categories of modules over a complete Noetherian local ring that satisfy Matlis duality —- a restriction we commonly rely on throughout the paper.  In that context, we introduce a duality between \subsel s (the \emph{smile dual}), and in Theorem~\ref{thm:smsdual}, we detail how properties of \subsel s relate across this duality.

Combining these approaches, in Section~\ref{sec:cidual}, we construct a duality between \resop s and \subsel s that restricts to a duality between \reseo\ closure operations on the one hand and interior operations on the other.  In Theorem~\ref{thm:clintdual} we prove this, along with an exploration of how several other properties translate across this duality.

Section~\ref{sec:testprerad} explores the general notion of test ideals.  In it, we show that a version of the theorems describing tight closure test ideals as a submodule of the injective hull of the residue field \cite{HHmain} and as a sum of images of maps to $R$ \cite{HaTa-gentest}
holds quite generally for residual closure operations arising from preradicals (i.e. \funcss\ \subsel s) on a category of artinian $R$-modules.  See Theorem~\ref{thm:preradical}.

Section~\ref{sec:exact} is devoted to exactness properties on preradicals.  We show in Proposition~\ref{pr:dualhered} that left exactness is dual to surjection-preservation.

In Section~\ref{sec:lim}, we introduce the notions of \emph{direct} and \emph{inverse limits} of submodule selectors, and we show (see Propositions~\ref{pr:dirlim} and \ref{pr:invlim}) that many good properties are preserved by our duality.  This framework becomes useful later, for example in Corollary \ref{cor:H0Idual}, Proposition \ref{pr:tordiv}, and Proposition \ref{pr:mixedchardual}.

The next few sections apply the above framework to various special cases.

Section~\ref{sec:trace} develops the notions of trace and module torsion, both indexed by \emph{pairs} of modules, as smile-duals to each other (see Theorem~\ref{thm:tracedual}). We show that our general notion of trace behaves well with respect to flat base change (see Theorem~\ref{thm:traceloc}).  We relate these notions to module closures (see Remark~\ref{rmk:modclosure}), zeroth local cohomology, and $I$-adic completion (see Corollary~\ref{cor:H0Idual}).  In Proposition~\ref{pr:testidealmc}, we find an application of Theorem~\ref{thm:preradical} to module closures by finite modules.

Section~\ref{sec:tordiv} applies our framework to notions of torsion (with respect to a multiplicative set) and divisibility.  Namely, it turns out that the $W$-torsion preradical is smile-dual to the $W$-divisible preradical (see Proposition~\ref{pr:tordiv}, where we also identify the relevant properties of these preradicals).

Section~\ref{sec:tcic} shows how our framework applies to the inspiration for our study, namely tight closure, (liftable) integral closure, and their associated test ideals.

Section~\ref{sec:almost} connects our framework to some mixed characteristic operations of P\'erez and the second named author in \cite{PeRG}, meant to provide a mixed characteristic version of tight closure theory.  Our insistence on indexing both trace and module torsion by a \emph{pair}, as well as our development of limits of submodule selectors, are particularly relevant here.

Section~\ref{sec:coloc}, the final section of our paper, concerns localization and the style of ``colocalization'' favored by K. Smith and A. Richardson, particularly in relation to smile-duality on systems of \subsel s.  As a result, we show that despite the fact \cite{BM-unloc} that tight closure of finite modules does not commute with localization, we have that tight closure in \emph{Artinian} modules commutes with \emph{colocalization} (see Theorem~\ref{thm:tccoloc}).  We use this to shed light on the result of Lyubeznik and Smith \cite[Theorem 7.1]{LySm-Test} that formation of the big test ideal commutes with localization and completion (see Corollary~\ref{cor:tcccp} and its proofs).

\section{Submodule selectors and residual operations}\label{sec:first}

In this section we define and give the basic properties of submodule selectors and residual operations, which will be the fundamental objects of study in this paper. These definitions are inspired by and generalize the notions of closure operation and interior operation. We work as generally as possible in this section, so that we may choose which additional assumptions to work with in later sections.

For the current section, $R$ is an arbitrary ring with identity (not necessarily commutative), and all modules are left $R$-modules.

\begin{defn}
Let $\cM$ be a class of $R$-modules that is closed under taking submodules and quotient modules.  Let $\cP := \cP_\cM$ denote the set of all pairs $(L,M)$ where $M \in \cM$ and $L$ is a submodule of $M$ in $\cM$.

A \emph{\subsel} is a function $\alpha: \cM \rightarrow \cM$ such that \begin{itemize}
 \item $\alpha(M) \subseteq M$ for each $M \in \cM$, and
 \item For any isomorphic pair of modules $M, N \in \cM$ and any isomorphism $\phi: M \rightarrow N$, we have $\phi(\alpha(M)) = \alpha(\phi(M))$.
  \end{itemize}
 We say that a \subsel\ $\alpha$ is \begin{itemize}
  \item \emph{order-preserving} if for any $(L,M) \in \cP$, we have $\alpha(L) \subseteq \alpha(M)$;
  \item \emph{\surfuncss} if for any surjective map $\pi: M \twoheadrightarrow N$ in $\cM$, we have $\pi(\alpha(M)) \subseteq \alpha(N)$;
  \item \emph{\funcss} if $\alpha$ is order-preserving and \surfuncss -- \emph{i.e.} if for any $g: M \rightarrow N$ in $\cM$, we have $g(\alpha(M)) \subseteq \alpha(N)$;
  \item \emph{\idemss} if for any $M \in \cM$ we have $\alpha(\alpha(M)) = \alpha(M)$.
    \item \emph{\radss} if for any $M \in \cM$, we have $\alpha(M/\alpha(M)) = 0$;
  \item an \emph{interior operation} if it is \idemss\ and order-preserving.
 \end{itemize}
  \end{defn}

\begin{defn}  
 An \emph{\extop} on $\cM$ is a function $e: \cP \rightarrow \cM$ such that \begin{itemize}
  \item For any $(L,M) \in \cP$, we have $L \subseteq e(L,M) \subseteq M$, and 
  \item For any isomorphic pair of modules $M,N \in \cM$, any isomorphism $\phi: M \rightarrow N$, and any submodule $L$ of $M$, we have $\phi(e(L,M)) = e(\phi(L),N)$.
  \end{itemize}
  When $e$ is an \extop, we use the standard notation for closure operations $L^e_M := e(L,M)$.  Hence, the second condition above becomes $\phi(L^e_M) = \phi(L)^e_N$.
  
We say that an \extop\ $e$ on $\cM$ is \begin{itemize}
 \item \emph{\reseo} if for any surjective map $q: M \twoheadrightarrow P$ in $\cM$, we have $(\ker q)^e_M = q^{-1}(0^e_P)$ (because $q$ is a surjection, we also have $q( (\ker q)_M^e ) = 0_P^e$);
 \item \emph{\opsub} if whenever $L \subseteq M \subseteq N$ are in $\cM$ with $(L,N), (M,N) \in \cP$, we have $L^e_N \subseteq M^e_N$.
 \item \emph{\opamb} if whenever $L \subseteq M \subseteq N$ are in $\cM$ with $(L,M), (L,N) \in \cP$, we have $L^e_M \subseteq L^e_N$.
 \item \emph{\idemeo} if for any $(L,M) \in \cP$, $(L^e_M)^e_M = L^e_M$.
 \item \emph{\surfunceo} if whenever $(L,M) \in \cP$ and $\pi: M \onto N$ a surjection in $\cM$, we have $\pi(L^e_M) \subseteq \pi(L)^e_N$.
 \item \emph{\funceo} if whenever $(L,M) \in \cP$ and $g: M \ra N$ is a map in $\cM$, we have $g(L^e_M) \subseteq g(L)^e_N$.
 \item a \emph{closure operation} if it is \opsub\ and \idemeo.
 \end{itemize}
The residual property will take on special import for us, so that we will use the expression \emph{\resop} to mean a residual \extop. 
\end{defn}

Next we construct a map $\rho$ that takes as input a submodule selector and outputs a residual operation, and a map $\sigma$ that takes as input a residual operation and outputs a submodule selector.

\begin{construction}\label{cons:res}
If $\alpha$ is a \subsel\ on $\cM$, define $\rho(\alpha): \cP \ra \cM$ by \[L^{\rho(\alpha)}_M := \pi^{-1}(\alpha(M/L)),\] where $\pi: M \twoheadrightarrow M/L$ is the natural surjection. 

Conversely, let $r$ be a \resop\ on $\cM$.  Define $\sigma(r): \cM \ra \cM$ by \[\sigma(r)(M) := 0^r_M.\]  
\end{construction}

\begin{prop}\label{pr:resid}
The function $\sigma$ given in Construction~\ref{cons:res} gives a bijection between \resop s on $\cM$ and \subsel s on $\cM$, with inverse given by $\rho$.
\end{prop}

\begin{proof}
First we show that if $\alpha$ is a \subsel\ on $\cM$, then $\rho(\alpha)$ is a \resop\ on $\cM$.  To see that it is an \extop, let $L \subseteq M$ be in $\cM$ and let $\pi: M \onto M/L$ be the natural surjection.  Then $L = \pi^{-1}(0) \subseteq \pi^{-1}(\alpha(M/L)) = L^{\rho(\alpha)}_M \subseteq \pi^{-1}(M/L) = M$.  Moreover, the isomorphism condition on \subsel s translates properly via $\rho$ to the isomorphism condition on \extop s.  To see that $\rho(\alpha)$ is residual, let $q: M \onto P$ be a surjective map in $\cM$.  Let $L = \ker q$.  Let $j: P \rightarrow M/L$ be the isomorphism such that $\pi := j \circ q: M \onto M/L$ is the natural surjection.  Then \begin{align*}
(\ker q)^{\rho(\alpha)}_M &= L^{\rho(\alpha)}_M = \pi^{-1}(\alpha(M/L)) \\
&= q^{-1}(j^{-1}(\alpha(M/L))) = q^{-1}(j^{-1}(\alpha(j(P))))\\
&= q^{-1}(j^{-1}(j(\alpha(P)))) =q^{-1}(\alpha(P)) = q^{-1}(0^{\rho(\alpha)}_P),
\end{align*}
where the first equality in the third line follows from the isomorphism condition on \subsel s.  Thus, $\rho(\alpha)$ is a \resop\ on $\cM$.

Conversely, if $r$ be a \resop\ (or any \extop) on $\cM$, then $\sigma(r)$ is clearly a \subsel\ on $\cM$.

Finally, we show that these operations are inverses of each other.  Accordingly, let $r$ be a \resop\ on $\cM$.  Let $L \subseteq M$ be a submodule inclusion in $\cM$, and $\pi: M \onto M/L$ the natural surjection.  Then \[
L^{\rho(\sigma(r))}_M = \pi^{-1}(\sigma(r)(M/L)) = \pi^{-1}(0^r_{M/L}) = L^r_M,
\]
where the last equality follows from the residual condition.  On the other hand, let $\alpha$ be a \subsel\ on $\cM$.  Then \[
\sigma(\rho(\alpha))(M) = 0^{\rho(\alpha)}_M = \alpha(M),
\]
since the natural surjection $M \onto M$ is just the identity map.
\end{proof}

For the next results, we will find it convenient to use the following general lemma about commutative squares.

\begin{lemma}\label{lem:square}
Suppose that the following represents a commutative diagram of sets and functions:
\[
\begin{CD}
M @>{\pi}>> P \\
@V{i}VV @V{j}VV \\
N @>{q}>> Q \\
\end{CD}
\]
Then for any subset $P' \subseteq P$, $i(\pi^{-1}(P')) \subseteq q^{-1}(j(P'))$.
\end{lemma}

\begin{proof}
We have $q(i(\pi^{-1}(P'))) = j(\pi(\pi^{-1}(P'))) \subseteq j(P')$. Applying $q^{-1},$ we get 
\[i(\pi^{-1}(P')) \subseteq q^{-1}(q(i(\pi^{-1}(P')))) \subseteq q^{-1}(j(P')).\]
\end{proof}

Next we show which properties of submodule selectors and residual operations are preserved by $\rho$ and $\sigma$. In Sections \ref{sec:ssdual} and \ref{sec:cidual} we will apply this result to our duality between closure operations and interior operations.

\begin{prop}\label{pr:ssres}
Let $\alpha$ be a \subsel\ on $\cM$ with $r=\rho(\alpha)$ under the bijection given in the previous proposition (so that $\alpha = \sigma(r)$).  Then: \begin{enumerate}
\item\label{it:resop} $\alpha$ is order-preserving if and only if $r$ is \opamb,
\item\label{it:ressf} The following are equivalent: \begin{enumerate}
 \item\label{itt:rops} $r$ is \opsub,
 \item\label{itt:ssf} $\alpha$ is \surfuncss\,
 \item\label{itt:rsf} $r$ is \surfunceo,
 \end{enumerate}
\item\label{it:resrad} $\alpha$ is \radss\ if and only if $r$ is \idemeo, and
\item\label{it:residem} $\alpha$ is \idemss\ if and only if $r$ satisfies $L_{L_M^r}^r = L_M^r$ for all $(L,M) \in \cP$.
\end{enumerate}
\end{prop}

\begin{proof}[Proof of (\ref{it:resop})]
Suppose $\alpha$ is order-preserving.  Let $L \subseteq M \subseteq N$ be submodule inclusions in $\cM$.  Let $p: N \onto N/L$ be the canonical surjection, and $\pi$ its restriction to $M \to M/L$.  Then $\alpha(M/L) \subseteq \alpha(N/L)$ by assumption, so we have \[
L^r_M = \pi^{-1}(\alpha(M/L)) = p^{-1}(\alpha(M/L)) \subseteq p^{-1}(\alpha(N/L)) = L^r_N.
\]
Conversely, suppose $r$ is \opamb.  Let $L \subseteq M$ be a submodule inclusion in $\cM$.  Then $\alpha(L) = 0^r_L \subseteq 0^r_M = \alpha(M)$.
\end{proof}
\begin{proof}[Proof of (\ref{it:ressf})]
(\ref{itt:rops}) $\implies$ (\ref{itt:ssf}): Suppose $r$ is \opsub.  Let $\pi: M \onto P$ be a surjection in $\cM$.  Without loss of generality, $P = M/L$ where $L=\ker(\pi)$ and $\pi$ is canonical.  By assumption, we have $0^r_M \subseteq L^r_M$.  Hence, $\pi(\alpha(M)) = \pi(0^r_M) \subseteq \pi(L^r_M) = \pi(\pi^{-1}(\alpha(M/L))) = \alpha(M/L) = \alpha(P)$.

(\ref{itt:ssf}) $\implies$ (\ref{itt:rsf}): Suppose $\alpha$ is \surfuncss.  Let $\pi:M \onto P$ be a surjection in $\cM$.  Let $(L,M) \in \cP$.  Then consider the following commutative diagram: \[
\begin{CD}
M @>{\pi}>> P \\
@V{p}VV @V{q}VV \\
M/L @>{\overline{\pi}}>> P/\pi(L) \\
\end{CD}
\]
Using Lemma~\ref{lem:square}, we have: $\pi(L^r_M) = \pi(p^{-1}(\alpha(M/L))) \subseteq q^{-1}(\overline{\pi}(\alpha(M/L))) \subseteq q^{-1}(\alpha(P/\pi(L))) = \pi(L)^r_P$.

(\ref{itt:rsf}) $\implies$ (\ref{itt:rops}): Suppose $r$ is \surfunceo.  Let $L \subseteq M \subseteq N$ be submodule inclusions with $(L,N), (M,N) \in \cP$.  Consider the natural map $q: N \onto N/M$.  Then \[
\frac{L^r_N + M}{M} = q(L^r_N) \subseteq q(L)^r_{N/M} = 0^r_{N/M} = q( M_N^r )=\frac{M^r_N}{M}.
\]
Hence, $L^r_N \subseteq M^r_N$.

\end{proof}

\begin{proof}[Proof of (\ref{it:resrad})]
Suppose $\alpha$ is \radss.  Let $(L,M) \in \cP$, and let $q$ denote the natural surjection $M \to M/L_M^r$.  We have \[
\begin{aligned}
0 &= \alpha\left(\frac{M/L}{\alpha(M/L)}\right) = \alpha\left(\frac{M/L}{0_{M/L}^r}\right)=\alpha\left(\frac{M/L}{\pi(L_M^r)}\right)=\alpha\left(\frac{M/L}{L_M^r/L}\right)=\alpha\left(\frac{M}{L_M^r}\right) \\
&= q(q^{-1}(\alpha(M/L_M^r))) = q((L_M^r)_M^r) =  \frac{(L^r_M)^r_M}{L^r_M}.
\end{aligned}
\]
Hence $L^r_M = (L^r_M)^r_M$, so that $r$ is \idemeo.

Conversely suppose $r$ is \idemeo.  Let $M \in \cM$, and let $\pi:M \to M/0_M^r$ be the natural surjection.  Then $0^r_M = (0^r_M)^r_M$, so \[\alpha(M/\alpha(M)) = \alpha(M/0^r_M) = \pi(\pi^{-1}(\alpha(M/0_M^r)))=\pi((0_M^r)_M^r)=
(0^r_M)^r_M / 0^r_M = 0.\]  Thus $\alpha$ is \radss.
\end{proof}

\begin{proof}[Proof of (\ref{it:residem})]
We have $\alpha(\alpha(M))=\alpha(0_M^r)=0_{0_M^r}^r$,

Suppose $\alpha$ is idempotent. We will show that $L_{L_M^r}^r=L_M^r$ for all $(L,M) \in \cP$. Since $\alpha$ is idempotent, $0_{0_M^r}^r=0_M^r$ for all $M \in \cM$. Let $(L,M) \in \cP$. Then $L_M^r=\pi^{-1}(0_{M/L}^r)$, where $\pi:M \to M/L$ is the quotient map. So if $p:L^r_M \to L^r_M/L$ is the restriction of $\pi$ to $L_M^r$, we have
\[L_{L_M^r}^r=p^{-1}(0^r_{L^r_M/L}).\]
But $L_M^r=\pi^{-1}(0_{M/L}^r)$. So $L_M^r/L=\pi(L_M^r)=\pi(\pi^{-1}(0_{M/L}^r))=0_{M/L}^r$, since $\pi$ is a surjection.
Hence 
\[L_{L_M^r}^r=p^{-1}(0^r_{L^r_M/L})=p^{-1}(0^r_{0_{M/L}^r})=p^{-1}(0_{M/L}^r). \]
Since $p$ is the restriction of $\pi$ to $L_M^r$, this is $\pi^{-1}(0_{M/L}^r) \cap L_M^r=L_M^r \cap L_M^r=L_M^r$, as desired.

Now suppose that $L_{L_M^r}^r=L_M^r$ for all $L \subseteq M$. In particular, we have $0_{0_M^r}^r=0_M^r$ for each $M \in \cM$.  This implies that $\alpha$ is idempotent.
\end{proof}

\section{Duality between submodule selectors}\label{sec:ssdual}
In this section, we define a duality operation $\dual$ between submodule selectors and show which properties from Section \ref{sec:first} correspond under this duality. We then give a useful alternate characterization of the duality.

Throughout this section, $R$ is a complete Noetherian local ring with maximal ideal $\m$, residue field $k$, and $E := E_R(k)$ the injective hull. We will use $^\vee$ to denote the Matlis duality operation. $\cM$ is a category of $R$-modules closed under  taking submodule and quotient modules, and such that for all $M \in \cM$, $M^{\vee\vee} \cong M$. So for example $\cM$ could be the category of \fg\ $R$-modules, or of Artinian $R$-modules.

\begin{defn}
Let $S(\cM)$ denote the set of all submodule selectors on $\cM$.  Define $\dual: S(\cM) \rightarrow S(\cM^\vee)$ as follows: For $\alpha \in S(\cM)$ and $M\in \cM^\vee$, \[
\alpha^\dual(M) := (M^\vee / \alpha(M^\vee))^\vee,
\]
considered as a submodule of $M$ in the usual way.  
\end{defn}

It is clear that the isomorphism-preserving property of $\alpha$ will be inherited by $\alpha^\dual$.  Moreover, we have the following.

\begin{thm}\label{thm:smsdual}
Let $\alpha$ be a submodule selector on $\cM$.  Then: \begin{enumerate}
 \item\label{it:duality} $(\alpha^\dual)^\dual = \alpha$,
 \item\label{it:funcdual} $\alpha$ is \surfuncss\ if and only if $\alpha^\dual$ is order-preserving, and
 \item\label{it:idemraddual} $\alpha$ is \idemss\ if and only if $\alpha^\dual$ is \radss.
\end{enumerate}
\end{thm}

\begin{proof}
First we show that $\alpha^{\dual\dual} = \alpha$.  For this, we have \[
(\alpha^\dual)^\dual(M) = (M^\vee / \alpha^\dual(M^\vee))^\vee = (M^\vee / (M/\alpha(M))^\vee)^\vee = (\alpha(M)^\vee)^\vee = \alpha(M).
\]

Next, suppose $\alpha$ is \surfuncss.  Let $L \subseteq M$ be a submodule inclusion in $\cM^\vee$, and let $\pi: M^\vee \onto L^\vee$ be the associated surjection in $\cM$.  Since $\pi(\alpha(M^\vee)) \subseteq \alpha(L^\vee)$ by assumption on $\alpha$, we have the induced surjections \[
\frac{M^\vee}{\alpha(M^\vee)} \onto \frac{\pi(M^\vee)}{\pi(\alpha(M^\vee))} = \frac{L^\vee}{\pi(\alpha(M^\vee))} \onto \frac{L^\vee}{\alpha(L^\vee)}.
\]
Dualizing the composite then induces an injective map $\alpha^\dual(L) \hookrightarrow \alpha^\dual(M)$, and since everything was natural, it follows that $\alpha^\dual(L) \subseteq \alpha^\dual(M)$.

For the converse, suppose $\beta = \alpha^\dual$ is order-preserving.  Let $\pi: M \onto P$ be a surjection in $\cM$.  Without loss of generality, $P=M/L$ and $\pi$ is the canonical surjection, where $L$ is a submodule of $M$.  Taking Matlis duals, we have that $(M/L)^\vee$ is a submodule of $M^\vee$.  Then \begin{align*}
\left(\frac{M}{\pi^{-1}(\alpha(M/L))}\right)^\vee &\cong \left(\frac{M/L}{\alpha(M/L)}\right)^\vee = \alpha^\dual((M/L)^\vee) = \beta((M/L)^\vee) \\
&\subseteq \beta(M^\vee)= \alpha^\dual(M^\vee) = \left(\frac{M}{\alpha(M)}\right)^\vee,
\end{align*}
where the containment in the second line above arises from the order-preservation property of $\beta$.  By naturality of the isomorphism above, it follows that the dual surjection $M/\alpha(M) \onto M/\pi^{-1}(\alpha(M/L))$ is also canonical, so that $\alpha(M) \subseteq \pi^{-1}(\alpha(M/L))$.  Applying $\pi$ to both sides, it follows that $\pi(\alpha(M)) \subseteq \pi(\pi^{-1}(\alpha(M/L))) = \alpha(M/L)$.  Thus, $\alpha$ is \surfuncss.

For (\ref{it:idemraddual}), first suppose $\alpha$ is \idemss\ and let $M \in \cM^\vee$.  Then \[
\alpha^\dual(M/\alpha^\dual(M)) = \left(\frac{\alpha(M^\vee)}{\alpha(\alpha(M^\vee))}\right)^\vee = 0^\vee = 0.
\]

Finally, suppose $\beta=\alpha^\dual$ is \radss.  Let $M \in \cM$.  Then \begin{align*}
\alpha(\alpha(M)) &= \beta^\dual(\beta^\dual(M)) = \beta^\dual\left(\left(\frac{M^\vee}{\beta(M^\vee)}\right)^\vee\right)\\
&= \left(\frac{M^\vee / \beta(M^\vee)}{\beta(M^\vee / \beta(M^\vee))}\right)^\vee
= (M^\vee / \beta(M^\vee))^\vee \\&= \beta^\dual(M) = \alpha(M)
\end{align*}
where the second equality of the second line comes from the fact that $\beta$ is \radss.
\end{proof}

Given a submodule selector $\alpha$, we give an alternate description of $\alpha^\dual$ that will be useful later on.

\begin{thm}
\label{alternatedualdescription}
Let $\alpha$ be a submodule selector and $M$ an $R$-module. Then 
\[\alpha^\dual(M)=\{z \in M : g(z)=0 \text{ for all } g \in \alpha(M^\vee)\}
\]
\end{thm}

\begin{proof}
We have an exact sequence
\[0 \to \alpha(M^\vee) \to M^\vee \to M^\vee/\alpha(M^\vee) \to 0,\]
whose Matlis dual is
\[0 \leftarrow \alpha(M^\vee)^\vee \xleftarrow{q} M^{\vee\vee} \leftarrow \alpha^\dual(M) \leftarrow 0.\]
Note that $q$ is the restriction map.

Set $i$ to be the standard isomorphism $M \cong M^{\vee\vee}$, and let $\mu:M \to \alpha(M^\vee)^\vee$ be $q \circ i$. This implies that
\[\alpha^\dual(M)=\left(M^\vee/\alpha(M^\vee)\right)^\vee = \ker(M \xrightarrow{\mu} (\alpha(M^\vee))^\vee) \subseteq M.\]

Note that for $z \in M$ and $g:M \to E$, $i(z)(g)=g(z)$. So 
\[(q \circ i)(z)=i(z) \mid_{\alpha(M^\vee)},\]
where $i(z):M^\vee \to E$.

An element $z \in \ker(\mu)=\ker(q \circ i)$ if and only if $i(z)\mid_{\alpha(M^\vee)}=0$, if and only if $\alpha(M^\vee) \subseteq \ker(i(z))$, if and only if for every $g \in \alpha(M^\vee)$, $g(z)=0$, if and only if $z \in \bigcap_{g \in \alpha(M^\vee)} \ker(g)$.
\end{proof}

\section{Closure-interior duality}\label{sec:cidual}
Using the results of Section \ref{sec:ssdual}, we achieve our original goal of a duality between residual closure operations and interior operations such as that between tight closure and tight interior \cite{nmeSc-tint}, and we indicate the properties that correspond via this duality.

In this section, $R$ is a complete local Noetherian ring.

\begin{defn}
If $r$ be a residual operation on $\cP$ (resp. $\cP^\vee$), set $i(r) := \sigma(r)^\dual$.  If $j$ is a submodule selector on $\cM$ (resp. $\cM^\vee$), set $c(j):= \rho(j^\dual)$. Here ${}^\dual$ is as in \S\ref{sec:ssdual} and $\rho, \sigma$ are as in Construction~\ref{cons:res}.
\end{defn}

In the special cases where $r$ is restricted   to be a \emph{closure} operation, or $j$ an \emph{interior} operation, the above gives the duality we desire, as outlined in Lemma~\ref{lem:cidual} and Theorem~\ref{thm:clintdual} below.

\begin{lemma}\label{lem:cidual}
We have that $c$ and $i$ are inverses of one another.  Hence, there is a one-to-one correspondence between residual operations on $\cP$ and submodule selectors on $\cM^\vee$.
\end{lemma}

\begin{proof}
This follows from Proposition~\ref{pr:resid} and Theorem~\ref{thm:smsdual}(\ref{it:duality}).
\end{proof}

\begin{thm}\label{thm:clintdual}
Let $r$ be a residual operation on $\cP$, and $\int := i(r)$.  Then: \begin{enumerate}
 \item $r$ is \opsub\ $\iff$ $\int$ is order-preserving.
 \item $r$ is \idemeo\ $\iff$ $\int$ is \idemss.
 \item $r$ is a closure operation $\iff$ $\int$ is an interior operation.
 \item Suppose $r$ is a closure operation (i.e. $\int$ is an interior operation). The following are equivalent: \begin{enumerate}
   \item $r$ is functorial.
	\item $r$ is \opamb.
	\item $\int$ is \surfuncss.
	\item $\int$ is functorial.
   \end{enumerate}
 \end{enumerate}
\end{thm}

\begin{proof}
This follows from Proposition~\ref{pr:ssres} and Theorem~\ref{thm:smsdual}.  Here are the details:

(1): $r$ is \opsub $\iff \sigma(r)$ is \surfuncss\ (by Proposition~\ref{pr:ssres}(2)) $\iff \int$ is order-preserving (by Theorem~\ref{thm:smsdual}).

(2): $r$ is \idemeo $\iff \sigma(r)$ is \radss (by Proposition~\ref{pr:ssres}(3)) $\iff \int$ is \idemss\ (by Theorem~\ref{thm:smsdual}(1 \& 3)).

(3) follows from (1) and (2).

(4): $r$ is functorial $\iff$ it is order-preserving on ambient modules (by Proposition~\ref{pr:ssres}(2)) $\iff \sigma(r)$ is order-preserving (by Proposition~\ref{pr:ssres}(1)) $\iff \int$ is \surfuncss (by Theorem~\ref{thm:smsdual}(2)) $\iff \int$ is functorial (by definition of interior operation).
\end{proof}

Summarizing the correspondences given in the above  sections, we have the following chart:

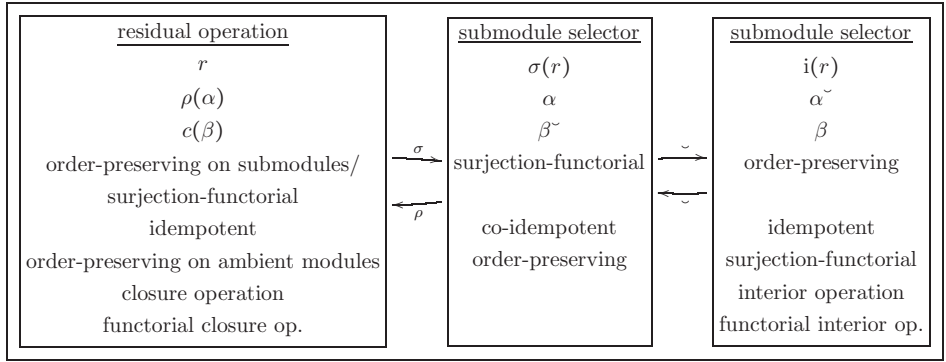
\[
\resizebox{\textwidth}{!}{
\boxed{
\xymatrix{
{\boxed{\begin{gathered}
\text{\underline{\resop}}\\
r\\
\rho(\alpha)\\
c(\beta)\\
\text{\opsub/} \\ \text{\surfunceo}\\ 
\text{\idemeo}\\
\text{\opamb}\\
\text{closure operation}\\
\text{functorial closure op.}
\end{gathered}}} \ar@/^1pc/@{->}[r]^(.62)\sigma&
{\boxed{\begin{gathered}
\text{\underline{submodule selector}}\\
\sigma(r)\\
\alpha\\
\beta^\dual\\
\text{\surfuncss}\\
\ \\
\text{\radss}\\
\text{order-preserving}\\
\ \\
\ 
\end{gathered}}} \ar@/^1pc/[l]^(.38)\rho \ar@/^1pc/[r]^\dual&
\boxed{\begin{gathered}
\text{\underline{submodule selector}}\\
\int(r)\\
\alpha^\dual\\
\beta\\
\text{order-preserving}\\
\ \\
\text{\idemss}\\
\text{\surfuncss}\\
\text{interior operation}\\
\text{functorial interior op.}
\end{gathered}} \ar@/^/[l]^\dual
}}
}\]

\section{Test ideals for preradicals}\label{sec:testprerad}
In this section we restrict to the case of preradicals, defined below, and give a more explicit form for the dual of a preradical and its finitistic version. We also give a result (Theorem \ref{thm:preradical})  that is a generalization of key properties of the tight closure test ideal originally described in \cite{HHmain}.

In this section, $R$ will be an arbitrary ring with identity (not necessarily commutative), and all modules are left $R$-modules.

\begin{defn}
\begin{enumerate}
\item A \emph{preradical} is a functorial submodule selector \cite{DikTho-closure}.

\item If $\alpha$ is a submodule selector on $\cM$, then the \emph{finitistic version} $\alpha_f$ is the submodule selector on $\cM$ given for each $M \in \cM$ by \[
\alpha_f(M) := \sum \{\alpha(L) \mid L \subseteq M, L \text{ \fg}, L \in \cM \}.
\]
\end{enumerate}
\end{defn}

We present the finitistic version $\alpha_f$ of $\alpha$ as a generalization of notions of finitistic tight closure and test ideals. 

\begin{lemma}\label{lem:finitistic}
Let $\alpha$ be a submodule selector on $\cM$. Then: \begin{enumerate}
\item $\alpha_f$ is order-preserving, 
\item if $\alpha$ is order-preserving, then $\alpha(M) = \alpha_f(M)$ for all finitely generated $M \in \cM$, and
\item if $\alpha$ is \surfuncss, then $\alpha_f$ is a preradical on $\cM$.
\end{enumerate}
\end{lemma}

\begin{proof}
Let $L \subseteq M$ be a submodule inclusion in $\cM$.  Let $x \in \alpha_f(L)$.  Then there is some finitely generated submodule $L'$ of $L$ (hence also of $M$) with $L' \in \cM$ and $x\in \alpha(L')$.  Thus, $x\in \alpha_f(M)$.

Now suppose $\alpha$ is order-preserving and let $M \in \cM$ be finitely generated.  It is clear that $\alpha(M) \subseteq \alpha_f(M)$ since $M$ is a finitely generated submodule of $M$ in $\cM$.  For the other containment, let $x\in \alpha_f(M)$.  Then there is a finitely generated submodule $L$ of $M$ with $x\in \alpha(L)$.  But $\alpha(L) \subseteq \alpha(M)$ since $\alpha$ is order-preserving.  So $x\in \alpha(M)$.

Now suppose $\alpha$ is \surfuncss.  Let $g: M \ra N$ be a module map in $\cM$.  Let $x\in \alpha_f(M)$.  Then there is a \fg\ submodule $M' \subseteq M$, $M' \in \cM$, with $x\in \alpha(M')$.  Let $N' := g(M')$, and let $g': M' \onto N'$ be the induced map.  Then since $\alpha$ is \surfuncss, $g(x) = g'(x) \in \alpha(N')$.  Since $N' \in \cM$ and $N'$ is a \fg\ submodule of $N$, we have $g(x) \in \alpha_f(N)$.
\end{proof}

\begin{lemma}\label{lem:dspr}
Let $M, N \in \cM$ and let $\alpha$ be a preradical on $\cM$.  Then $\alpha(M \oplus N) = \alpha(M) \oplus \alpha(N)$ as submodules of $M \oplus N$.
\end{lemma}

\begin{proof}
Let $(x,y) \in \alpha(M \oplus N)$.  Let $\pi_1$, $\pi_2$ be the canonical projections of $M \oplus N$ onto $M$, $N$ respectively.  Then since $\alpha$ is a preradical, $x = \pi_1(x,y) \in \alpha(M)$ and $y = \pi_2(x,y) \in \alpha(N)$.  Thus, $(x,y) \in \alpha(M) \oplus \alpha(N)$.

Conversely, let $(x,y) \in \alpha(M) \oplus \alpha(N)$.  Let $i_1$, $i_2$ be the canonical injections of $M,N$ respectively into $M \oplus N$.  Then since $\alpha$ is a preradical, $(x,0) = i_1(x) \in \alpha(M \oplus N)$ and $(0,y) = i_2(y) \in \alpha(M \oplus N)$.  Hence $(x,y) = (x,0) + (0,y) \in \alpha(M \oplus N)$.
\end{proof}

The following result was already known (see, for example \cite{Di-clCM}), but we include it to show how it follows naturally from the results of this paper.

\begin{cor}\label{cor:dscl}
Let $\cl$ be a \reseo\ functorial closure operation on $\cP$. Let $U \subseteq L$, $V \subseteq M$ be submodule inclusions in $\cP$.  Then \[
U^\cl_L \oplus V^\cl_M = (U \oplus V)^\cl_{L \oplus M}.
\]
\end{cor}

\begin{proof}
Let $\alpha := \sigma(\cl)$.  Then by Lemma~\ref{lem:dspr}, $\alpha((L/U) \oplus (M/V)) = \alpha(L/U) \oplus \alpha(M/V)$.  Hence, using the standard identification $\frac LU \oplus \frac MV = \frac{L \oplus M}{U \oplus V}$, we have 
\[ 
\begin{aligned}
U_L^{\cl} \oplus V_M^{\cl} &=\pi_1^{-1}(\alpha(L/U)) \oplus \pi_2^{-1}(\alpha(M/V)) \\
&=\pi^{-1}(\alpha(L/U) \oplus \alpha(M/V)) \\
&=\pi^{-1}(\alpha(L/U \oplus M/V)) = \pi^{-1}\left(\alpha\left(\frac{L \oplus M}{U \oplus V}\right)\right)\\
&=(U \oplus V)_{L \oplus M}^{\cl},
\end{aligned}
\]
where $\pi:L \oplus M \to (L \oplus M)/(U \oplus V)$ is the quotient map, and $\pi_1:L \to L/U$ and $\pi_2:M \to M/V$ are its component maps.
\end{proof}

The following result was inspired by the main results of \cite{PeRG}. Our result requires the ring to be complete and restricts the module category, but allows arbitrary preradicals (rather than just those coming from module closures). In the case of tight closure, this is a classical result, originally given as Proposition 8.23 in \cite{HHmain} for the finitistic version and discussed in the lecture notes of October 26th in \cite{HoFNDTC}.

\begin{thm}\label{thm:preradical}
Let $(R,\m,k)$ be a commutative complete Noetherian local ring and $E$ the injective hull of its residue field.  Let $\alpha$ be a preradical on the category of \emph{Artinian} $R$-modules $\cA$.  Let $\alpha_f$ be the finitistic version.  Then:
\begin{align*}
\alpha^\dual(R) &= \ann (\alpha(E)) = \bigcap_{M \in \cA} \ann(\alpha(M))
\\ &\subseteq \alpha_f^\dual(R) = \ann(\alpha_f(E)) 
\\ &= \bigcap_{\lambda(M)< \infty} \ann(\alpha(M)) \subseteq \bigcap_{\lambda(R/I) < \infty} \ann (\alpha(R/I)).
\end{align*}

Moreover, the last containment is an equality whenever $R$ is approximately Gorenstein (e.g. reduced, or depth at least 2).

If $\alpha$ represents a residual closure operation $\cl$ defined on artinian modules and on \emph{all} ideals, in such a way that for any ideal $J$, there is a collection $\{I_\lambda\}_{\lambda \in \Lambda}$ of $\m$-primary ideals for some index set $\Lambda$, such that $J = \bigcap_\lambda I_\lambda$ and $J^\cl = \bigcap_\lambda I_\lambda^\cl$.  Then $\bigcap_{\lambda(R/I) < \infty} \ann (\alpha(R/I)) = \bigcap_{I \textnormal{ ideal of } R} (I : I^{\cl})$.
\end{thm}

\begin{proof}
We prove all the required containments.

To see that $\alpha^\dual(R) = \ann (\alpha(E))$, recall that under the identification $R = E^\vee$, the dual of the exact sequence $0 \ra \alpha(E) \ra E \ra E/\alpha(E) \ra 0$ becomes $0 \ra (E/\alpha(E))^\vee \ra R \arrow{\pi} \alpha(E)^\vee \ra 0$, where $\pi(r)(x) := rx$.  Thus, we have 
$\alpha^\dual(R) = (E/\alpha(E))^\vee = \ker \pi 
= \{r \in R \mid rx = 0 \text{ for all } x\in \alpha(E)\} = \ann (\alpha(E))$.

The fact that $\alpha_f^\dual(R) = \ann (\alpha_f(E))$ follows from replacing $\alpha$ by $\alpha_f$ in the above, using the fact from Lemma~\ref{lem:finitistic} that $\alpha_f$ is a preradical.

To prove that $\ann (\alpha(E)) \subseteq \bigcap_{M \in \cA} \ann(\alpha(M))$, we show that $\ann (\alpha(E))$ annihilates $\alpha(M)$ for all artinian $M$. Note that for any artinian $M$, we have that $M^\vee$ is finitely generated.  Hence, there is a surjection $p: R^{\oplus t} \onto M^\vee$ for some positive integer $t$.  Dualizing, we have $p^\vee: M \hookrightarrow E^{\oplus t}$.  Now, $\ann (\alpha(E)) = \ann (\alpha(E)^{\oplus t}) = \ann(\alpha(E^{\oplus t}))$ by Lemma~\ref{lem:dspr}.  But since $M$ embeds as a submodule of $E^{\oplus t}$ and $\alpha$ is order-preserving, we have $\alpha(M) \hookrightarrow \alpha(E^{\oplus t})$.  So any element of $R$ that annihilates $\alpha(E^{\oplus t})$ annihilates $\alpha(M)$ as well.

The reverse containment $\ann (\alpha(E)) \supseteq \bigcap_{M \in \cA} \ann(\alpha(M))$ follows from the fact that $E$ is an artinian $R$-module.

Next we show that $\alpha_f(E)$ is the common annihilator of all modules $\alpha(M)$ for $M$ finite length.  For this, let $r \in \ann (\alpha_f(E))$.  Then for any finite-length $R$-module $M$, we have that $M$ is artinian and $\alpha(M) = \alpha_f(M)$. Hence by the above, $r \in \ann (\alpha(M))$.  Conversely, let $r \in \bigcap_{\lambda(M)<\infty} \ann (\alpha(M))$.  Let $x\in \alpha_f(E)$.  Then there is a finitely generated (hence finite length) submodule $M$ of $E$ with $x \in \alpha(M)$.  Thus $rx=0$, completing the proof of the current claim.

Next we prove that  $\alpha^\dual(R) \subseteq \alpha_f^\dual(R)$. For any $r \in \alpha^\dual(R)$ and any finite-length $M$, we have that $M$ is artinian, so $r \alpha(M) = 0$.  Hence $r\in \alpha_f^\dual(R)$.  Thus, $\alpha^\dual(R) \subseteq \alpha_f^\dual(R)$.

The last displayed containment is clear.

Now suppose $R$ is approximately Gorenstein.  Let $M$ be a finite length $R$-module.  By Lemma~\ref{lem:dspr}, we may assume $M$ is $\oplus$-indecomposable.  Let $J = \ann (M)$.  Since $J$ is $\m$-primary and $R$ is approximately Gorenstein, there is an irreducible $\m$-primary ideal $I$ with $I \subseteq J$.  Then $M$ is a direct sum of indecomposable finite-length modules over the Artinian Gorenstein ring $R/I$.  Hence $M \hookrightarrow E_{R/I}(k) = R/I$. 

For the final statement, let $\cl$ be a closure operation satisfying the given conditions.  The containment $\supseteq$ follows from the fact that $\alpha(R/I) = (I : I^\cl)$ for any ideal $I$ of finite colength.  Conversely, let $r\in R$ such that $r$ annihilates $\alpha(R/I)$ for every ideal $I$ of finite colength.  Let $J$ be an arbitrary ideal and $x\in J^\cl$.  Let $\{I_\lambda\}_{\lambda \in \Lambda}$ be a collection of ideals as in the hypothesis.  Then suppose $r \in \alpha(R/I)$ for all $\m$-primary $I$.  Then for any $x\in J^\cl$, we have $x\in I_\lambda^\cl$ for all $\lambda$, whence $rx \in \bigcap_\lambda I_\lambda = J$.
\end{proof}

\begin{rem}
\label{rem:tcicmprimary}
Recall that the last condition above holds for both tight closure and (liftable) integral closure.  Namely, $I^* = \bigcap_n (I + \m^n)^*$ \cite[Theorem 1.5(4)]{HuTC} and $I^- = \bigcap_n (I+\m^n)^-$ \cite[Corollary 6.8.5]{HuSw-book} respectively.
\end{rem}

\section{Exactness properties for preradicals}\label{sec:exact}

Since a preradical is an additive functor (see Lemma~\ref{lem:dspr}), it is natural to ask when it preserves exactness.  Of course injectivity is preserved by definition.  Accordingly, we recall the following results, for which $R$ will be an arbitrary ring with identity (not necessarily commutative), and all modules left $R$-modules.

\begin{defpr}[Hereditary and cohereditary preradicals {\cite[6.9]{LiftingModules}}]
\label{def:heredcohered}
Let $\alpha$ be a preradical on $\cM$.  Then the following are equivalent \begin{enumerate}
\item $\alpha$ is left exact.
\item For any submodule inclusion $L \subseteq M$ in $\cM$, we have $\alpha(L) = L \cap \alpha(M)$.
\item $\alpha$ is \idemss\ and whenever $L \subseteq M$ is a submodule inclusion in $\cM$ with $\alpha(M) = M$, we have $\alpha(L) = L$.
\end{enumerate}
In this case, we say that $\alpha$ is \emph{hereditary}.

On the other hand, the following are equivalent: \begin{enumerate}
\item $\alpha$ preserves surjections, i.e. given a surjection $\pi:M \to N$ in $\cM$, $\pi(\alpha(M))=\alpha(N)$.
\item For any submodule inclusion $L \subseteq M$ in $\cM$, we have $\alpha(M/L) = \frac{L+\alpha(M)}{L}$.
\item $\alpha$ is \radss\ and whenever $L \subseteq M$ is a submodule inclusion in $\cM$ with $\alpha(M) = 0$, we have $\alpha(M/L) = 0$ as well.
\end{enumerate}
In this case, we say that $\alpha$ is \emph{cohereditary}.
\end{defpr}

\begin{rem}
The identity preradical is both hereditary and cohereditary.  The same holds for the preradical that sends every module to its zero submodule.

Note also that hereditary is equivalent to left exact, but cohereditary is weaker than right exact.
\end{rem}

\begin{cor}
Let $\alpha$ be a preradical on $\cM$.  Then $\alpha$ is an exact functor if and only if it is hereditary and cohereditary.
\end{cor}

\begin{rem}
From the point of view of closure operations, such preradicals appear quite scarce.  Indeed, both hereditariness and cohereditariness fail for many common closure operations, such as tight closure and liftable integral closure.  However, as we see in the next result, these two properties are dual to each other:
\end{rem}

\begin{prop}\label{pr:dualhered}
Let $R$ be a commutative complete local Noetherian ring and let $\alpha$ be a preradical on $\cM^\vee$.  Then $\alpha^\dual$ is hereditary if and only if $\alpha$ is cohereditary.
\end{prop}

\begin{proof}
Let $L \subseteq M$ be a submodule inclusion.  Then $\alpha^\dual(M) \cap L = \alpha^\dual(L)$ if and only if the induced map $L/\alpha^\dual(L) \rightarrow M/\alpha^\dual(M)$ is injective.  Taking Matlis duals, the above map is injective if and only if  the map 
\[\alpha(L^\vee) = \alpha^{\dual\dual}(L^\vee) = (L / \alpha^\dual(L))^\vee \leftarrow (M/\alpha^\dual(M))^\vee = \alpha(M^\vee)\] is surjective.  But every surjection in $\cM^\vee$ occurs as a map of the form $M^\vee \onto L^\vee$, where $L \subseteq M$ is a submodule inclusion in $\cM$.
\end{proof}

\begin{example}
The casual reader might think that the duality ${}^\dual$ we are using makes left exactness dual to right exactness.  But this is not true, as the following class of examples shows.

Let $I$ be an ideal of a commutative complete local Noetherian ring, and let $\sigma(M) := \ann_M(I)$ for all $M \in \cM$.  Let $\tau = \sigma^\dual$; then $\tau(N) = IN$ for all $N \in \cM^\vee$.  We will explore these preradicals further in \S\ref{sec:trace}.  It is easily seen both that $\sigma$ is left exact and that $\tau=\sigma^\dual$ is not right exact (though it is surjection-preserving).  Indeed, it is not even exact in the middle.  For a counterexample, take any submodule inclusion $L \subseteq M$ with $IL \neq IM \cap L$.
\end{example}

\section{Limits of \subsel s}\label{sec:lim}

We begin the current section by imposing a binary relation $\leq$ on the collection of all \subsel s on $\cM$, as follows: We say $\alpha \leq \beta$ if $\alpha(M) \subseteq \beta(M)$ for all $M \in \cM$.  This is easily seen to be a partial order. We will discuss direct and inverse limits of appropriate posets of submodule selectors, and show that the dual operation reverses the type of limit.

In this section, let $R$ be a ring with identity (not necessarily commutative), and all modules are left $R$-modules.

\begin{defn}
Let $\Gamma$ be a directed poset (i.e., for all $i,j \in \Gamma$, there is a $k \in \Gamma$ such that $i,j \le k$), and $\{s_j\}_{j \in \Gamma}$ a set of \subsel s such that if $i \le j$, $s_i \leq s_j$. Define $(\displaystyle \varinjlim_{j \in \Gamma} s_j)(M):=\bigcup_{j \in \Gamma} s_j(M)=\sum_{j \in \Gamma} s_j(M).$
\end{defn}

\begin{prop}\label{pr:dirlim}
$s := \displaystyle \varinjlim_{j \in \Gamma} s_j$ is a \subsel.  Moreover, it is order-preserving (resp. \surfuncss, resp. functorial) if all the $s_j$ have the corresponding property.  Further, assuming that all the $s_j$ are functorial, $s$ is hereditary if all the $s_j$ are too, and if the $s_j$ are cohereditary, then $s$ is also cohereditary, and in particular \radss.
\end{prop}

\begin{proof}
To see that $s$ is a \subsel, let $g: M \rightarrow N$ be an isomorphism in $\cM$, and let $x\in s(M)$.  Then there is some $j\in \Gamma$ such that $x\in s_j(M)$.  Since $g$ is an isomorphism and $s_j$ is a \subsel, $g(x) \in s_j(N)$.  Hence, $g(x) \in \bigcup_i s_i(N) = s(N)$.

Now suppose all the $s_j$ are order-preserving.  Let $L \subseteq M$ be a submodule inclusion in $\cM$.  Let $x\in s(L)$.  Then there is some $j$ with $x\in s_j(L)$.  Since $s_j$ is order-preserving, $x\in s_j(M)$.  Hence $x\in s(M)$.

Next suppose all the $s_j$ are \surfuncss.  Let $\pi: L \onto M$ be a surjection in $\cM$.  Let $x\in s(L)$.  Then there is some $j$ with $x\in s_j(L)$.  Since $s_j$ is \surfuncss, $\pi(x) \in s_j(M)$.  Hence $\pi(x)\in s(M)$.

Now suppose all the $s_j$ are functorial and hereditary.  Let $L \subseteq M$ be a submodule inclusion in $\cM$.  Let $x \in s(M) \cap L$.  Then there is some $j$ with $x\in s_j(M) \cap L = s_j(L)$ (the latter since $s_j$ is hereditary), whence $x \in s(L)$.

Suppose the $s_j$ are cohereditary. Then in particular they are \radss\ (see Definition/Proposition \ref{def:heredcohered}). Let $L \subseteq M$ be a submodule inclusion in $\cM$. We first prove that $s$ is \radss.

We have
\[s(M/s(M))=\sum_j s_j\left( \frac{M}{\sum_j s_j(M)}\right).\]
Since the $s_j$ are \radss, $s_j(M/s_j(M))=0$ for all $j \in \Gamma$. Since the $s_j$ are cohereditary, this implies that 
\[s_j\left( \frac{M}{\sum_j s_j(M)}\right)=0\]
for all $j \in \Gamma$. This implies that their sum is 0, and so $s(M/s(M))=0$.

Now we prove that if $s(M)=0$, then $s(M/L)=0$. We have 
\[s(M/L)=\sum_{j} s_j(M/L)=0,\]
since $0=s(M)=\sum_j s_j(M)$ and the $s_j$ are cohereditary. Hence $s$ is cohereditary.
\end{proof}

\begin{defn}
Let $\Omega$ be an inverse poset and $\{t_j\}_{j \in \Omega}$ a set of \subsel s on $\cM$ such that if $i \leq j$, then $t_i\leq t_j$. Define $\displaystyle \varprojlim_{j \in \Omega} t_j$ by 
\[\left(\displaystyle \varprojlim_j t_j\right)(M)=\bigcap_{j \in \Omega} t_j(M).\]
\end{defn}

\begin{prop}\label{pr:invlim}
$t := \displaystyle \varprojlim_{j \in \Omega} t_j$ is a \subsel.  Moreover, it is order-preserving (resp. \surfuncss, resp. \funcss, resp. hereditary) if all the $t_j$ have the corresponding property.
\end{prop}

\begin{proof}
To see that $t$ is a \subsel\ (resp. order-preserving, resp. \surfuncss) when all the $t_j$ have the corresponding property, let $g: L \rightarrow M$ be an isomorphism (resp. injection, resp surjection) in $\cM$.  Let $x\in t(L)$.  Then for all $j$, $x\in t_j(L)$.  By the given property for all the $t_j$ and for $g$, we have $g(x) \in t_j(M)$ for all $j$.  Hence $g(x) \in \bigcap_j t_j(M) = t(M)$.

Suppose all of the $t_j$ are hereditary, i.e. for any submodule inclusion $L \subseteq M$ in $\cM$, $t_j(L)=t_j(M) \cap L$. Then
\[t(L)=\bigcap_{j} t_j(L)=\bigcap_j (t_j(M) \cap L)=\left(\bigcap_j t_j(M)\right) \cap L=t(M) \cap L,\]
so $t$ is also hereditary.
\end{proof}

Next, we interface with our notion of duality. For the rest of this section, $R$ is a commutative complete local Noetherian ring.

\begin{prop}\label{pr:dualreverse}
Let $\alpha$, $\beta$ be \subsel s on $\cM$ with $\alpha \leq \beta$.  Then $\alpha^\dual \geq \beta^\dual$.
\end{prop}

\begin{proof}
Let $M \in \cM^\vee$.  Then $\alpha(M^\vee) \subseteq   \beta(M^\vee)$, so that $M^\vee / \alpha(M^\vee) \onto M^\vee / \beta(M^\vee)$.  Applying Matlis duality, it follows that $\beta^\dual(M) = (M^\vee / \beta(M^\vee))^\vee \subseteq (M^\vee / \alpha(M^\vee))^\vee = \alpha^\dual(M)$.
\end{proof}

\begin{defn} If $\Gamma$ is a poset, let $\Gamma'$ denote its poset dual. \end{defn}

\begin{prop}\label{pr:duallim1}
Let $\Gamma$ be a directed poset and $\{s_j\}_{j \in \Gamma}$ a directed system of submodule selectors. Then $\displaystyle \bigg(\varinjlim_{j \in \Gamma} s_j\bigg)^\dual = \displaystyle \varprojlim_{j \in \Gamma'} (s_j^\dual)$.
\end{prop}

\begin{proof}
$\Gamma'$ is an inverse poset, and $\{s_j^\dual\}_{j \in \Gamma}$ is an inverse system of submodule selectors. Set $t_j=s_j^\dual$. Let $s=\varinjlim_{j \in \Gamma} s_j$ and let $z \in M$. Then $z \in s^\dual(M)$ if and only if for all $g \in s(M^\vee)$, $z \in \ker(g)$. This holds if and only if for all $g \in \cup_j s_j^\dual(M^\vee)$, $z \in \ker(g)$, if and only if for all $j$, $g \in s_j(M^\vee)$, $z \in \ker(g)$. This is true if and only if for all $j$, $z \in s_j^\dual(M)$, $z \in \cap_j s_j^\dual(M)=(\varprojlim_j s_j^\dual)(M)$.
\end{proof}

\begin{prop}\label{pr:duallim2}
Let $\Omega$ be an inverse poset and  $\{t_j\}_{j \in \Omega}$ an inverse system of submodule selectors. Then $\displaystyle\bigg(\varprojlim_j t_j\bigg)^\dual=\varinjlim_{j \in \Omega'} (t_j^\dual)$.
\end{prop}

\begin{proof}
Write $s_j=t_j^\dual$. Then $\{s_j\}_{j \in \Omega'}$ form a direct limit system. So $(\varprojlim_j t_j)^\dual=(\varprojlim_j (s_j^\dual))^\dual$. By Proposition~\ref{pr:duallim1}, this is equal to $((\varinjlim_j s_j)^\dual)^\dual=\varinjlim_j s_j=\varinjlim_j t_j^{\dual}$.
\end{proof}

Note that if $N$ is an $R$-module and $\{N_j\}_{j \in \Gamma}$ is a directed system of $R$-modules such that $\varinjlim N_j=N$, it is not always the case that $\varprojlim t_{N_j}=t_N$ when $t$ is the trace map defined in Section \ref{sec:trace}. So there are examples where $t_N^\dual \ne \varinjlim t_{N_j}^\dual$. See Example \ref{ex:tracebreakslimits}.

We do get one inclusion when $t$ is trace: $t_N \le \varprojlim_\alpha t_{N_j}$. See Proposition \ref{pr:directedtrace}.

\section{Special cases: trace, torsion, completion, and module closures}\label{sec:trace}

In this section we apply the structure of the preceding sections to traces, torsion submodules, and module closures, all examples of preradicals and residual operations that appear elsewhere in the literature. The results on the trace are inspired by work of the second named author in \cite{PeRG}.

\subsection*{Traces}

Let $R$ be a commutative ring with identity.

\begin{defn}
Let $L$ be an $R$-module. Then the \emph{$L$-trace} of an $R$-module $N$ is 
\[\tr_L(N)=\im(L \otimes \Hom_R(L,N) \to N),\]
where $\ell \otimes f \mapsto f(\ell)$. Equivalently, $\tr_L(N)$ is the submodule of $N$ generated by the set
\[\{f(\ell) \mid f \in \Hom(L,N), \ell \in L\},\]
or
\[\tr_L(N) = \sum_{f \in \Hom_R(L,N)} f(L).\]

When $N=R$, this is known as the trace ideal of $L$ \cite{Lam99,Lin17}.

Given a subset $S \subseteq L$, the \emph{$(S,L)$-trace} $\tr_{S,L}(N)$ of $N$ is the $R$-submodule of $N$ generated by the set \[
\{f(s) \mid f \in \Hom(L,N), s \in S\}.
\]  That is, \[
\tr_{S,L}(N) = \sum_{f \in \Hom_R(L,N)}  Rf(S).
\]

Hence, $\tr_L = \tr_{L,L}$. If $S = \{x\}$ is a singleton, we write $\tr_{x,L} := \tr_{\{x\},L}$.

The ideal $\tr_{x,L}(R)$ is known as the order ideal of $x \in L$ \cite{EvGr-Ord}.

If the ring need be specified, we note it in the superscript, as follows: \[\tr_{X,L}^{(R)}.\]
\end{defn}

\begin{rem}\label{rem:tracefacts}
We note the following easy facts.
\begin{enumerate}
    \item For any $R$-module $L$ and any subset $S \subseteq L$, $\tr_{S,L}$ is a submodule selector on $\cM$.  In particular, $\tr_{x,L}$ and $\tr_L = \tr_{L,L}$ are submodule selectors.
    \item When $L=A$ is an $R$-algebra, $\tr_A = \tr_{1,A}$.
    \item When $I$ is an ideal of $R$, $\tr_{R/I}(N) = (0 :_N I)$.
    \item It is elementary that for any $z\in N$, $Rz \cong R/\ann_R(z)$.  Hence by (3), $\tr_{Rz}(N) = \ann_N (\ann_R(z))$.
    \item When $I$ is an ideal of $R$ and $L,N$ are $R$-modules, we have $\tr_{IL,L}(N) = I\tr_L(N)$.  In particular, $\tr_{I,R}(N) = IN$.
    \item If $L$ is an $R$-module and $x\in L$, then $\tr_{x,L} \leq \tr_{Rx}$: if $z\in \tr_{x,L}(M)$, then there is some $R$-linear $g: L \rightarrow M$ with $g(x)=z$; composing this with the inclusion $i: Rx \rightarrow M$, we have $z=(g\circ i)(x)$, so that $z\in \tr_{Rx}(M)$.
    \item However, it is not always the case that $\tr_{Rx} \le \tr_{x,L}$.  Let $R=M=k[\![t]\!]$, $x=1$, and $L=k(\!(t)\!)$, the fraction field of $R$.  Then since $\Hom_R(L,M)=0$, we have $\tr_{x,L}(M)=0$.  But since $Rx$ is a free $R$-module, we have $\tr_{Rx}(M) = M$.
\end{enumerate}
\end{rem}

\begin{lemma}\label{lem:trsums}
Let $L$ be an $R$-module, $S \subseteq L$ a subset, and $RS$ the submodule of $L$ generated by $S$. Then: \begin{enumerate}[(a)]
    \item $\tr_{S,L} = \tr_{RS,L}$, and 
    \item Given $M \in \cM$ and $z\in M$, we have $z\in \tr_{S,L}(M)$ if and only if there exists a positive integer $k$, $R$-linear maps $g_i: L \ra M$, and elements $s_i\in S$ such that $z=\sum_{i=1}^k g_i(s_i)$.
\end{enumerate}
\end{lemma}

\begin{proof}
Since $S \subseteq RS$, $\tr_{S,L} \le \tr_{RS,L}$. Let $z \in \tr_{RS,L}(M)$. Then there exist $g_i:L \to M$, elements $z_i \in RS$, and $r_i \in R$ such that $z=\sum_i r_ig_i(z_i)$. Each $z_i=\sum_j a_{ij}s_j$ where the $s_j \in S$. Hence
\[z=\sum_{i,j} a_{ij}r_ig_i(s_j) \in \tr_{S,L}(M).\]

For (b), the ``if'' direction holds by definition.  Conversely, let $z \in \tr_{S,L}(M)$. Then there exist $r_i \in R$, $h_i:L \to M$ $R$-linear, and $s_i \in S$ such that $z=\sum_{i=1}^k r_ih_i(s_i)$. Define $g_i:L \to M$ by $g_i(x)=r_ih_i(x)$. Then the $g_i$ are $R$-linear maps $L \to M$ and
\[z=\sum_{i=1}^k g_i(s_i).
\qedhere \]
\end{proof}

\begin{lemma}\label{lem:algtrace}
Let $L$ be an $R$-module and $S$ a subset of $L$.  Then the \subsel\ $\tr_{S,L}$ is \funcss, i.e. a preradical, and $\tr_L$ is \idemss.
\end{lemma}

\begin{proof}
First we prove that $\tr_{S,L}$ is functorial. Let $f:M \to N$ be an $R$-linear map in $\cM$ and let $z \in \tr_{S,L}(M)$. Then by Lemma~\ref{lem:trsums}, there are $R$-linear maps $g_i:L \to M$ and elements $s_i \in S$ such that $z=\sum_i g_i(s_i)$. Then $f \circ g_i$ are $R$-linear maps from $L$ to $N$, and $\sum_i (f \circ g_i)(s_i)=f(z)$. Hence $f(z) \in \tr_{S,L}(N)$, as desired.

For idempotence, let $M \in \cM$ and $z \in \tr_L(M)$. Then there is a map $g:L \to M$ and an element $q \in L$ such that $g(q)=z$. The image of $g$ must be contained in $\tr_L(M)$, so we can view $g$ as a map $L \to \tr_L(M)$. Hence $z \in \tr_L(\tr_L(M)).$
\end{proof}

\begin{example}
In general, $\tr_{S,L}$ is not idempotent. For example, take $S$ to be an ideal $I$ and $L=R$. For any $R$-module $M$, we have $\tr_{I,R}(M)=IM$. In particular, $\tr_{I,R}(R)=I$, so this submodule selector is idempotent if and only if $I=I^2$.
\end{example}

\begin{rem}
\label{rem:localcohom}
Next, we apply the limits from Section~\ref{sec:lim} to traces.  In particular, given a map $g: L \rightarrow L'$ and $X \subseteq L$, then it is easy to see that $\tr_{g(X), L'} \leq \tr_{X,L}$.  In particular, if $S \rightarrow S'$ is a map of $R$-algebras, then $\tr_{S'} \leq \tr_S$.  Hence, an \emph{inverse} system of algebras leads to a \emph{direct} system of traces.  In particular, if $I$ is an ideal of $R$, then we have natural surjections $R/I^{n+1} \onto R/I^n$, and we have \[
(\varinjlim \tr_{R/I^n})(M) = \bigcup_{n\in \N} (0 :_M I^n) = H^0_I(M).
\]
\end{rem}

The following condition helps us relate different traces of the same module.

\begin{defn}
Let $L$ and $M$ be $R$-modules. We say that $L$ generates $M$ if some direct sum of copies of $L$ surjects onto $M$. In particular, $L$ generates $M$ if there is a surjection $L \twoheadrightarrow M$.
\end{defn}

The following lemma is well-known, but appears in particular in \cite[Proposition 2.8]{Lin17}.

\begin{lemma}
If $L \twoheadrightarrow M$ is a surjection of $R$-modules, or more generally if $L$ generates $M$, then $\tr_M(N) \le \tr_L(N)$.
\end{lemma}

As a consequence, we get the following result concerning traces with respect to chains of increasing cyclic modules.

\begin{lemma}
\label{lem:seqtrace}
Let $L$ be an $R$-module and $\{z_n\}_{n \in \N} \subseteq L$. If 
\[Rz_1 \subseteq Rz_2 \subseteq \ldots \subseteq Rz_n \subseteq \ldots,\]
then $\tr_{Rz_j} \le \tr_{Rz_{j+1}}$ for all $j \in \N$.
\end{lemma}

\begin{proof}
For all $j \in \N$, $\ann_R (z_j) \supseteq \ann_R (z_{j+1})$. Hence we have surjections 
\[R/\ann_R(z_{j+1}) \twoheadrightarrow R/\ann_R(z_j).\]
By Remark \ref{rem:tracefacts}, $\tr_{R_{z_j}}=\tr_{R/\ann_R(z_j)}$, and hence the result follows.
\end{proof}

The next result was mentioned at the end of Section \ref{sec:lim}.

\begin{prop}
\label{pr:directedtrace}
Let $\Gamma$ be a directed poset and $\{L_j\}_{j \in \Gamma}$ a set of $R$-modules such that $\tr_{L_i} \ge \tr_{L_j}$ for all $i \le j$. In particular, it may be the case that $L_i$ generates $L_j$ whenever $i \le j$, or $\{L_j\}$ may be a set of $R$-algebras, with $R$-algebra maps $L_i \to L_j$ whenever $i \le j$. Let $L=\varinjlim_j L_j$. Then $\tr_L \le \varprojlim_j \tr_{L_j}$.
\end{prop}

\begin{proof}
By the hypothesis, for $i \le j$, $\tr_{L_i} \ge \tr_{L_j}$. Hence $\tr_L \le \tr_{L_j}$ for all $j \in \Gamma$. This implies that $\tr_L \le \bigcap_j \tr_{L_j}$, which by definition is $\varprojlim_j \tr_{L_j}$.
\end{proof}

We end this subsection by noting the following, which amounts to a general principle regarding how to deal with traces when working with modules over various rings.

\begin{prop}\label{pr:trbc}
Let $T$ be an $R$-algebra, $L$ a $T$-module, and $X \subseteq L$ a subset.  Let $U$ be the $T$-submodule of $L$ generated by $X$.  Then $\tr_{X,L} = \tr_{U,L}$, where in both cases, $\tr$ means $\tr^{(R)}$.
\end{prop}

\begin{proof}
The fact that $\tr_{X,L} \leq \tr_{U,L}$ follows by definition and from the fact that $X$ is a subset of $U$.

Conversely, let $M$ be an $R$-module.  To show that $\tr_{U,L}(M) \subseteq \tr_{X,L}(M)$, it suffices to show that for any $R$-linear $g: L \rightarrow M$ and any $y \in U$, we have $g(y) \in \tr_{X,L}(M)$.  So take such $y$ and $g$.  Since $U$ is generated by $X$ as a $T$-module, there exist $n\in \N$, $t_i \in T$, and $x_i \in X$ such that $y= \sum_{i=1}^n t_i x_i$.

For each $1\leq i \leq n$, consider the $R$-linear map $h_i: L \rightarrow M$ given by $h_i(\ell) := g(t_i \ell)$.  This is $R$-linear because for any $\ell, \ell' \in L$ and $r\in R$, we have \[
h_i(\ell + r\ell') = g(t_i(\ell + r\ell')) = g(t_i \ell) + r g(t_i \ell') = h_i(\ell) + rh_i(\ell').
\]
Then $g(y)=\sum_{i=1}^n g(t_i x_i) = \sum_{i=1}^n h_i(x_i) \in \tr_{X,L}(M)$.
\end{proof}

\subsection*{Trace behaves well under flat base change}

The next result generalizes a result of Lindo \cite[Proposition 2.8(viii)]{Lin17} to the setting of $(X,L)$-trace, using a similar proof technique.

\begin{thm}\label{thm:traceloc}
Let $B$ be an $R$-module, $X$ a subset of $B$, and $S$ a commutative $R$-algebra. Then for any $R$-module $M$, 
\[\im(\tr_{X,B}(M)\otimes_R S \rightarrow M \otimes_R S) \subseteq \tr_{X', B\otimes_R S}(M \otimes_R S)\] as $S$-submodules of $M \otimes_R S$, where $X' = \{x \otimes s \mid x\in X, s\in S\} \subseteq B \otimes_R S$. If $B$ is finitely presented and
$S$ is flat over $R$, then we have equality. 
\end{thm}

\begin{proof}
Let $z\in \tr_{X,B}(M)$ and $s \in S$.  Then there exist $n\in \N$ and $R$-linear maps $f_i: B \rightarrow M$ and elements $x_i \in X$ for $1\leq i \leq n$ such that $z=\sum_{i=1}^n f_i(x_i)$.  Then the maps $f_i \otimes 1: B \otimes_R S \rightarrow M \otimes_R S$ are $S$-linear, and $z \otimes s = \sum_{i=1}^n (f_i \otimes 1)(x_i \otimes s)$. Hence, $\im(\tr_{X,B}(M)\otimes_R S \rightarrow M \otimes_R S) \subseteq \tr_{X', B\otimes_R S}(M \otimes_R S)$.

Now suppose $S$ is a flat $R$-algebra and $B$ a finitely generated $R$-module.  Let $c \in \tr_{X', B \otimes_R S}(M \otimes_R S)$.  Then there are $S$-linear maps $g_i: B\otimes_R S \rightarrow M \otimes_R S$ and elements $d_i \in X'$ such that $c = \sum_i g_i(d_i)$.  We have $d_i = \sum_k x_{ik} \otimes s_{ik}$ for elements $x_{ik} \in X$ and $s_{ik} \in S$.  Moreover, by the assumptions on $B$ and $S$, the following map is an isomorphism (see e.g. \cite[Lemma 3.2.4]{EnJe-relhombook}):
\[
\phi:\Hom_R(B,M) \otimes S \to \Hom_S(B \otimes_R S, M \otimes_R S)
\]
sending $f \otimes s \mapsto (b \otimes t \mapsto f(b) \otimes st)$.
For each $i$, 
\[g_i=\phi\left(\sum_j f_{ij} \otimes t_{ij}\right)\]
for $R$-linear maps $f_{ij}:B \to M$ and elements $t_{ij} \in S$.
We have
\begin{align*}
g_i(d_i) &=\phi\left(\sum_j f_{ij} \otimes t_{ij}\right)\left(\sum_{k} x_{ik} \otimes s_{ik}\right) \\
&=\sum_{j,k} \phi(f_{ij} \otimes t_{ij})(x_{ik} \otimes s_{ik}) \\
&=\sum_{j,k}(f_{ij}(x_{ik}) \otimes t_{ij}s_{ik}) \in \tr_{X,B}(M) \otimes_R S.    
\end{align*}
It follows that \[
c = \sum_i g_i(d_i) \in \tr_{X,B}(M) \otimes_R S. \qedhere
\]
\end{proof}

We include a statement of Lindo's result, which deals with the special case corresponding to the usual trace operation, for the sake of keeping the paper self-contained:

\begin{cor}[{\cite[Proposition 2.8(viii)]{Lin17}}]\label{cor:trloc}
If $B$ is a finitely presented $R$-module and $S$ is a flat commutative $R$-algebra, then $\tr_B(R) \otimes_R S = \tr_{B \otimes_R S}(S)$.  In particular, for any prime ideal $\p \in \Spec R$, we have $\tr_B(R)_\p = \tr_{B_\p}(R_\p)$ and ${\tr_B(R)} \widehat{R_\p} = \widehat{\tr_{B_\p}(R_\p)}$.
\end{cor}
\begin{example}
We give an example where $B$ is not \fg\ and the formation of the trace ideal does not commute with localization, even though the ring is complete. Let $R$ be a complete DVR, e.g. $k[[t]]$, let $\m=(t)$ be its maximal ideal, let $B$ be its fraction field, and let $P := (0)R$.
Then $\tr_B(R)=0$.  To see this, let $g: B \rightarrow R$ be $R$-linear and $\alpha = g(1)$.  Then for any positive integer $n$, we have $\alpha = g(1) = g(t^n \cdot 1/t^n) = t^n g(1/t^n) \in \m^n$.  Thus, $\alpha \in \bigcap_n \m^n = 0$,
whence $g=0$.  On the other hand, since $B_P=R_P$, $\tr_{B_P}(R_P)=R_P$.
\end{example}

\subsection*{Dual of trace is module-torsion}

We continue the convention that $R$ is a commutative ring with identity.

\begin{defn}
Let $L$ be an $R$-module, and $S \subseteq L$.  Then we define the \emph{module torsion with respect to $S \subseteq L$} to be the submodule selector given by \[
\tom_{S,L}(M) := \{z\in M \mid \forall s\in S,\ s \otimes z=0 \text{ in } L \otimes_R M \}.
\]
If $S= \{x\}$ is a singleton, we write $\tom_{x,L} := \tom_{\{x\},L}$.  If $S=L$, we write $\tom_L := \tom_{L,L}$.
\end{defn}

\begin{rem}\label{rem:tomfacts}
We note the following easy facts, parallel to Remark~\ref{rem:tracefacts}.
\begin{enumerate}
    \item For any $R$-module $L$ and any set $S \subseteq L$, $\tom_{S,L}$ is a submodule selector on $\cM$.  In particular, $\tom_{x,L}$ and $\tom_L$ are submodule selectors.
    \item When $L=A$ is an $R$-algebra, $\tom_A = \tom_{1,A}$.
    \item When $I$ is an ideal of $R$   and $L,N$ are $R$-modules, $\tom_{L/IL}(N) = \pi^{-1}(\tom_L(N/IN))$, where $\pi: N \onto N/IN$ is the natural surjection.  In particular, $\tom_{R/I}(N) = IN$.
    \item When $I$ is an ideal of $R$, $\tom_{I,R}(N) = (0 :_N I)$.
    \item If $L$ is an $R$-module and $x\in L$, then $\tom_{Rx} \leq \tom_{x,L}$: consider the composition $M \ra Rx \otimes_R M \ra L \otimes_RM$, where the first map sends $m \mapsto x \otimes_R m$.
    \item However, the converse is false.  Let $x$ be a regular element in $R$ and let $M=R/(x)$. Then $\tom_{Rx}(M)=\tom_{R/\ann_R(x)}(M)=\ann_R(x)M=0$ but $\tom_{x,R}(M)=\ann_M(x)=M$.
\end{enumerate}
\end{rem}

Next we give a kind of dual to Proposition~\ref{pr:trbc}.
\begin{prop}\label{pr:tombc}
Let $T$ be an $R$-algebra, $L$ a $T$-module, and $X \subseteq L$ a subset.  Let $U$ be the $T$-submodule of $L$ generated by $X$.  Then $\tom_{X,L} = \tom_{U,L}$, where in both cases, $\tom$ means $\tom^{(R)}$.
\end{prop}

\begin{proof}
The fact that $\tom_{U,L} \leq \tom_{X,L}$ follows by definition and from the fact that $X$ is a subset of $U$.

For the opposite direction, let $M$ be an $R$-module and $z \in \tom_{X,L}(M)$.  Let $u\in U$.  Then there exist $n\in \N$, $t_i \in T$, and $x_i \in X$ such that $u=\sum_{i=1}^n t_i x_i$.  Using the left $T$-module structure of $L \otimes_RM$, we have in that module \[
u \otimes z = (\sum_{i=1}^n t_i x_i) \otimes z = \sum_{i=1}^n t_i \cdot (x_i \otimes z) = \sum_{i=1}^n t_i \cdot 0 = 0. \qedhere
\]
\end{proof}

\begin{thm}\label{thm:tracedual}
Let $R$ be a complete local Noetherian ring, $L$ an $R$-module, and $S \subseteq L$.  Then $\tr_{S,L}^\dual = \tom_{S,L}$.
\end{thm}

\begin{proof}
Let $M$ be a Matlis-dualizable module.  We first let $z\in \tom_{S,L}(M)$.  That is, $z\in M$ such that $s\otimes z = 0$ in $L \otimes_RM$ for all $s\in S$. By Theorem~\ref{alternatedualdescription}, we want to show that $g(z)=0$ for all $g\in \tr_{S,L}(M^\vee)$.  So let $g\in \tr_{S,L}(M^\vee)$.  Then there exist $R$-linear maps $\phi_i: L \rightarrow M^\vee$ and $s_i \in S$ such that $g = \sum_{i=1}^n \phi_i(s_i)$.  Let $\psi_i: L \otimes_RM \rightarrow E$ be the maps corresponding to the $\phi_i$ by Hom-tensor adjointness.  Since $s_i \otimes z=0$ in $L \otimes M$, we have \[
g(z) = \sum_{i=1}^n \phi_i(s_i)(z) = \sum_{i=1}^n \psi_i(s_i \otimes z) = \sum_{i=1}^n \psi_i(0) = 0.
\]
This implies that $\tom_{S,L} \subseteq \left(\tr_{S,L}\right)^\dual$.

For the reverse inclusion, let $z\in \tr_{S,L}^\dual(M)$ and let $s\in S$.  If $s \otimes z \neq 0$ in $L \otimes_RM$, then since $\Hom_R(-,E)$ never kills a nonzero module, we have $(R \cdot (s \otimes z))^\vee \neq 0$.  Since $E$ is injective, it follows that there is an $R$-linear map $\phi: L \otimes_R M \rightarrow E$ such that $\phi(s \otimes z) \neq 0$. 

Let $\psi: L \rightarrow M^\vee$ be the corresponding map that arises from Hom-tensor adjointness.  Then we have \[
\psi(s)(z) = \phi(s \otimes z) \neq 0.
\]
On the other hand, since $z\in \tr_{S,L}^\dual(M)$ and $\psi(s) \in \tr_{S,L}(M^\vee)$, Theorem~\ref{alternatedualdescription} yields that $z\in \ker \psi(s)$, contradicting the display above. Hence, $s \otimes z =0$ in $L \otimes_R M$ for all $s\in S$, whence $z\in \tom_{S,L}(M)$, as was to be shown.
\end{proof}

Note that the $R$-module $L$ has no restriction on it--it does not have to be either \fg\ or artinian.

\begin{cor}
Let $R$ be a complete local ring, $L$ an $R$-module, and $S \subseteq L$. Then $\tom_{S,L}$ is functorial and $\tom_L$ is \radss.
\end{cor}

\begin{proof}
By Theorem \ref{thm:tracedual}, $\tom_{S,L}$ is dual to $\tr_{S,L}$. By Lemma \ref{lem:algtrace}, $\tr_{S,L}$ is functorial and $\tr_L$ is idempotent, and so by Theorems~\ref{thm:tracedual} and \ref{thm:smsdual}, $\tom_{S,L}$ is functorial and $\tom_L$ is \radss.
\end{proof}

\begin{rem}\label{rmk:modclosure}
Recall \cite[Definition 2.3]{RG-bCMsing} that for an $R$-module $L$, the \emph{module closure} given by $L$ is given by the formula \[
N^{\cl_L}_M := \{ u\in M \mid \forall x\in L,\ x \otimes u \in \im (L \otimes N \rightarrow L \otimes M) \},
\]
whenever $M$ is an $R$-module and $N$ a submodule of $M$.  In the complete local Noetherian case, it follows from Theorem~\ref{thm:tracedual} that in the notation of Section~\ref{sec:cidual}, we have $\sigma(\cl_L) = \tom_L$, and hence by the above theorem we have \[
i(\cl_L) = \tr_L.
\]
That is, \emph{the $L$-trace is the interior operation dual to the module closure defined by $L$}. This demonstrates how our framework can be used to achieve results comparable to those in \cite{PeRG}.
\end{rem}

\begin{cor}\label{cor:dualalgebratrace}
If $R$ is a complete local Noetherian ring and $Q$ is an $R$-algebra, then $\tr_Q^\dual(M)=\tom_Q(M) = \ker(M \to Q \otimes M)$, where the map sends $z \mapsto 1 \otimes z$. 
\end{cor}

\begin{cor}\label{cor:dualIM}
Let $I$ be an ideal of a complete local Noetherian ring $R$.  Then $\tr_{R/I}^\dual(M) = IM$.
\end{cor}

The following result connects $\tom$ to the 0th local cohomology module, similar to Remark \ref{rem:localcohom}.

\begin{cor}\label{cor:H0Idual}
Let $I$ be an ideal of a complete local Noetherian ring $R$.  Then for any $R$-module $M$, \begin{align*}
(H^0_I)^\dual(M) &= (\varinjlim \tr_{R/I^n})^\dual(M) = (\varprojlim (\tr_{R/I^n}^\dual))(M) \\
&= (\varprojlim (\tom_{R/I^n}))(M) = \bigcap_n (I^n M) = \ker (M \rightarrow \widehat{M}^I).
\end{align*}
\end{cor}

The next example demonstrates that direct limits do not commute with taking $\tom$.

\begin{example}
\label{ex:tracebreakslimits}
Let $(R,m)$ be a complete DVR (for example $\widehat{\Z_{(2)}}$ or $k[[x]]$), $\pi \in R$ such that $m=(\pi)$, $N_j=\frac{1}{\pi^j}R$, and $K=\Frac(R)$. Then $K=\varinjlim N_j$. 
Since $N_j \cong R$ for all $j$, $\tr_{N_j}(R) = R$. However, $\tr_K(R)=0$. So 
\[\left(\varprojlim_j \tr_{N_j}\right)(R)=\bigcap_j \left(\tr_{N_j}(R)\right)=\bigcap_j R=R,\]
but
\[\tr_{(\varinjlim_j N_j)}(R)=\tr_{K}(R)=0.\]
Hence $\tom_K \gneq \varinjlim \tom_{N_j}.$
\end{example}

\begin{prop}\label{pr:testidealmc}
Let $(R,\m)$ be a complete local Noetherian ring and $Q$ an $R$-module such that the quotient module $Q/JQ$ is $\m$-adically separated for any ideal $J$ of $R$.  Then the last condition of Theorem~\ref{thm:preradical} applies to $\tom_Q$, and hence to $\cl_Q$.  Thus, \[
\bigcap_{\lambda(R/I) < \infty} \ann (\tom_Q(R/I)) = \bigcap_{I  \textrm{ ideal of } R} (I : I^{\cl_Q}).
\]
In particular, this holds when $Q$ is \fg.
\end{prop}

\begin{proof}
Let $J$ be an ideal of $R$.  Then $J = \bigcap_n (J + \m^n)$ by the Krull intersection theorem, and if $a \in \bigcap_n (J+\m^n)^{\cl_Q}$, then for any $q\in Q$, we have $aq \in \bigcap_n (J +\m^n)Q = JQ$ by assumption of $\m$-adic separatedness of $Q/JQ$.  Hence $a\in J^{\cl_Q}$, whence $J^{\cl_Q} = \bigcap_n (J+\m^n)^{\cl_Q}$ as required.
\end{proof}

\section{Special cases: torsion and divisibility for a multiplicative set}\label{sec:tordiv}

As another application of our setup, we consider the notions of torsion and divisibility with respect to a multiplicative set.  These are known in a general way to be dual to each other; we clarify the nature of the duality in the context of Matlis duality over a complete local ring.  

In this section, $R$ is a commutative ring with identity.

Accordingly, recall the following definitions as given in \cite[\S4]{nmeYao-flat}:

\begin{defn}
Let $R$ be a ring, $W \subseteq R$ a multiplicatively closed set, and $M$ an $R$-module.  We say $M$ is \begin{itemize}
    \item \emph{$W$-torsion free} if for all $w\in W$, the homothety map on $M$ induced by $w$ is injective, or equivalently, if the natural map $M \rightarrow W^{-1}R \otimes_R M$ given by $z \mapsto \frac z1$ is injective.
    \item \emph{$W$-divisible} if for all $z\in M$ and all $w\in W$, there is some $y\in M$ such that $z=wy$, or equivalently if for all $w\in W$, the homothety map on $M$ induced by $w$ is surjective.
    \item \emph{h$_W$-divisible} if there is a free $W^{-1}R$-module $N$ and a surjective $R$-linear map $N \onto M$, or if equivalently (see \cite[Lemma 4.1]{nmeYao-flat}) the evaluation map $\Hom_R(W^{-1}R, M) \rightarrow M$ given by $f \mapsto f(1/1)$ is surjective.
    \end{itemize}
\end{defn}

 When used without modifier, typically $R$ is an integral domain and $W=R \setminus \{0\}$. Divisibility (and h-divisibility) are weak forms of injectivity, whereas torsion-freeness is a weak form of flatness, and in general are very important properties of modules.  Hence it behooves us to consider these notions in \subsel\ terms.

\begin{defn}
Let $W \subseteq R$ be a multiplicatively closed set, and $\cM$ a class of modules closed under taking submodules and quotient modules.  \begin{itemize}
    \item We define the \emph{$W$-torsion \subsel} $\tto_W$ by setting $\tto_W(M) := $ the kernel of the localization map $M \rightarrow W^{-1}M$, or equivalently, the set of $z\in M$ such that there exists $w\in W$ such that $wz=0$. Since $W^{-1}M \cong W^{-1}R \otimes_R M$, with the localization map corresponding to the map $z \mapsto 1 \otimes z$, it follows that $\tto_W = \tom_{W^{-1}R}$.
    \item We define the \emph{$W$-divisible \subsel} $\dv_W$ by setting $\dv_W(M) := $ the sum (hence the \emph{union}) of the $W$-divisible submodules of $M$.
\end{itemize}
\end{defn}

\begin{prop}\label{pr:tordiv}
Let $W$ be a multiplicatively closed set.  \begin{enumerate}
    \item An $R$-module $M$ is $W$-torsion free if and only if $\tto_W(M) =0$.
    \item An $R$-module $M$ is $W$-divisible if and only if $\dv_W(M) = M$.
    \item An $R$-module $M$ is h$_W$-divisible if and only if $\tr_{W^{-1}R}(M) = M$.
    \item For any $R$-module $M$, $\tr_{W^{-1}R}(M) \subseteq \dv_W(M)$.  That is, \[\tr_{W^{-1}R} \leq \dv_W.\]
    \item $\displaystyle \tto_W = \varinjlim_{w\in W} \tr_{R/(w)}$.
    \item $\dv_W$ is \idemss\ and functorial.
    \item $\tto_W$ is \idemss, functorial, and \radss.
    \item If $R$ is a complete local Noetherian ring and $M$ is Matlis-dualizable, then $\dv_W(M) = \tr_{W^{-1}R}(M)$.
    \item If $R$ is a complete local Noetherian ring, then when applied to Matlis-dualizable modules, we have $\tto_W = \dv_W^\dual$.
    \item In this case, $\dv_W(M) = \bigcap_{w\in W} wM$.
\end{enumerate}
\end{prop}

\begin{proof}
Part (1) is by definition.

For part (2), the ``only if'' direction is clear.  Conversely, suppose $\dv_W(M) = M$.  Let $z\in M$ and $w\in W$.  Then there is some $W$-divisible submodule $L$ of $M$ with $z\in L$.  Hence $z\in wL \subseteq wM$, whence the homothety is surjective.

Part (3) follows from the definition and Remark~\ref{rem:tracefacts}(2).

For part (4), let $z\in \tr_{W^{-1}R}(M)$ and $w\in W$.  Then there is some $R$-linear map $g: W^{-1}R \rightarrow M$ with $z=g(1/1)$.  Thus, $z=g(1/1)= g(w/w) = wg(1/w) \in wM$.

For part (5), first note that modules of the form $R/(w)$, $w\in W$ form an inverse limit system, since if $w, w' \in W$, then also $ww' \in W$, and we have $R/(ww') \onto R/(w)$ and $R/(ww') \onto R/(w')$ via the canonical maps.  Hence as in \S\ref{sec:trace}, one can discuss a direct limit of traces.  In particular, for any $R$-module $M$, we have \begin{align*}
\left(\varinjlim_{w\in W} \tr_{R/(w)}\right)(M) &= \bigcup_{w\in W} \tr_{R/(w)}(M) \\
&= \bigcup_{w\in W} (0 :_M w) = \tto_W(M).
\end{align*}

To see part (6), let $z\in \dv_W(M)$.  Then there is some $W$-divisible submodule $L$ of $M$ such that $z\in L$.  But since $L \subseteq \dv_W(M)$, it follows that $z\in \dv_W(\dv_W(M))$. For functoriality, first note that if $g: L \ra M$ is an $R$-linear map and $L'$ a $W$-divisible submodule of $L$, then $g(L')$ is $W$-divisible.  For if $u=g(t) \in g(L')$ and $w\in W$, then since $t\in L'$, there is some $s\in L'$ with $t=ws$, whence $u=g(t) = g(ws) = wg(s) \in wg(L')$.  Now let $z\in \dv_W(L)$. Then there is a $W$-divisible submodule $L'$ of $L$ with $z\in L'$.  But then $g(z) \in g(L')$, which is a $W$-divisible submodule of $M$, whence $g(z) \in \dv_W(M)$. 

To see part (7), let $z\in \tto_W(M)$.  Then there is some $w\in W$ with $wz =0$, and $z\in \tto_W(M)$; hence $z\in \tto_W(\tto_W(M))$.  For functoriality, let $g: L \ra M$ be a map in $\cM$ and $z\in \tto_W(L)$.  Then $wz=0$ for some $w\in W$, so $wg(z) = g(wz) = g(0)=0$, whence $g(z) \in \tto_W(M)$.  For co-idempotence, let $N=M/\tto_W(M)$, and $z\in M$ such that $\bar z\in \tto_W(N)$.  Then for some $w\in W$, $w\bar z = \overline{wz} =\bar 0$, whence $wz \in \tto_W(M)$.  Then there is some $v\in W$ with $vwz=0$, so that $z \in \tto_W(M)$, whence $\bar z = \bar 0$.  Hence $\tto_W(M/\tto_W(M)) = 0$.

For part (8), first recall \cite[Lemma 4.6]{nmeYao-flat} that in this context, a module is divisible if and only if it is h$_W$-divisible.  By part (4), we need only show that $\dv_W(M) \subseteq \tr_{W^{-1}R}(M)$.  Let $z\in \dv_W(M)$.  Since $\dv_W(M)$ is $W$-divisible (by (2) and (6)) and hence h$_W$-divisible (by the result quoted above), there is some $R$-linear map $g: W^{-1}R \ra \dv_W(M)$ with $z=g(1)$.  Let $j: \dv_W(M) \into M$ be the inclusion map.  Then $f:=j \circ g: W^{-1}R \ra M$ is $R$-linear with $z=f(1)$, so $z\in \tr_{W^{-1}R}(M)$.

For part (9), we need only quote part (8) and Corollary~\ref{cor:dualalgebratrace}.

Finally, part (10) follows from part (5), Corollary~\ref{cor:dualIM}, and Proposition~\ref{pr:duallim1}.
\end{proof}

\begin{rem}
It follows that when restricted to Matlis-dualizable modules over complete local Noetherian rings, $\dv_W$ and $\tto_W$ are \radss\ preradicals.  Thus they have associated \reseo\ closure operations.  Namely, \[
L^{\tto_W}_M = \pi^{-1}(\tto_W(M/L)) = \bigcup_{w\in W} (L :_M w)
\] and \[
L^{\dv_W}_M = \pi^{-1}(\dv_W(M/L)) = \bigcap_{w\in W} (L + wM),
\]
where $\pi: M \onto M/L$ is the canonical surjection.
\end{rem}

\section{Special case: tight closure and integral closure}
\label{sec:tcic}

Tight closure inspired much of the study of residual closure operations in commutative algebra. We apply our framework to tight closure and note that it is dual to tight interior as defined in \cite{nmeSc-tint}.

In this section, let $R$ be a commutative Noetherian ring with identity.

\begin{defn}
Let $R$ be a commutative Noetherian ring of prime \charp, $F^e$ denote the $e$th Frobenius functor, and $N \subseteq M$ be $R$-modules. Let $^eM$ denote the set of elements of $M$ with the same additive structure as $M$ and the $R-R$ bimodule structure given by $r \cdot (^ex) \cdot s={}^e(r^qsx)$ for $r,s \in R$ and ${}^ex \in {}^eM$. We define $N_M^{[q]}$ to be the kernel of the map $F^e(M) \to {F^e(M/N)}$ and $u^q:=1 \otimes_R u \in {}^eR \otimes_R M$. The \textit{tight closure of $N$ in $M$} is the set of $u \in M$ such that there exists a $c \in R^\circ$ with
\[cu^q \in N_M^{[q]}\]
for all $q=p^e \gg 0$ \cite[Section 8]{HHmain}.

Assume further that $R_{\text{red}}$ is $F$-finite. For $q_0$ a power of $p$ and $c \in R^\circ$, let
\[M_*[c,q_0]:=\sum_{q \ge q_0} \im(\Hom_R(^e(R),M) \to M),\]
where the map sends $g:{}^e R \to M$ to $g({}^e c)$. The \textit{tight interior of $M$}, denoted $M_*$, is
\[M_*:=\bigcap_{c \in R^\circ} \bigcap_{q_0 \ge 1} M_*[c,q_0]\]
\cite{nmeSc-tint}.
\end{defn}

\begin{prop}\label{pr:tint}
Let $R$ be a complete local Noetherian ring of \charp\ and let $\alpha(M)=0_M^*$, where $*$ denotes tight closure. 
\begin{enumerate}
    \item $\alpha$ is \funcss\ and \radss.
    \item If $R$ is complete local and reduced, then viewing $\alpha$ as a submodule selector on \fg\ and Artinian $R$-modules, $\alpha^\dual$ is tight interior.
    \item Tight interior is \funcss\ and \idemss.
\end{enumerate}
\end{prop}

\begin{proof}
\begin{enumerate}
    \item \cite[Section 8]{HHmain} tells us that tight closure is functorial and idempotent. The result then follows from Proposition \ref{pr:ssres}.
    \item \cite[Corollary 4.6]{nmeSc-tint}.
    \item \label{tightinteriorprops} These are in \cite{nmeSc-tint}: \funcss\ is Lemma 2.1, \idemss\ is Proposition 2.7, and localization is Corollary 2.11.
\end{enumerate}
\end{proof}

We recover the following well-known result about tight closure from our framework.

\begin{prop}
Let $(R,m)$ be a complete local Noetherian ring of \charp. Then tight closure satisfies the hypotheses of Theorem \ref{thm:preradical}, and hence its conclusions.
\end{prop}

\begin{proof}
This follows from Remark \ref{rem:tcicmprimary} and Theorem \ref{thm:preradical}.
\end{proof}

Another commonly used closure operation is integral closure (see, for example \cite{HuSw-book}). We use here the residual version of integral closure, defined in \cite{nmeUlr-lint}:

\begin{defn}[{\cite{nmeUlr-lint}}]
Let $L \subseteq F$ be $R$-modules, where $F$ is free. Let $S$ be the symmetric algebra over $R$ defined by $F$, with the natural grading, and let $T$ be the subring of $S$ induced by the inclusion $L \subseteq F$. Note that $S$ is $\N$-graded and generated in degree 1 over $R$, and $T$ is an $\N$-graded subring of $S$, also generated over $R$ in degree 1. The \textit{integral closure of $L$ in $F$}, denoted $L_F^-$ is the degree 1 part of the integral closure of the subring $T$ of $S$.

Now let $L \subseteq M$ be $R$-modules and $\pi:F \to M$ a surjection of a free module $F$ onto $M$. Let $K=\pi^{-1}(L)$. The \textit{liftable integral closure of $L$ in $M$} is
\[L_M^\li := \pi(K_F^\li).\]
\end{defn}

\begin{rem}
In \cite[Proposition 2.2]{nmeUlr-lint}, Epstein and Ulrich show that this definition is independent of the choice of free module $F$ and surjection $\pi:F \to M$.
\end{rem}

\begin{prop}
\label{prop:liftableintegralclosureprops}
Liftable integral closure is residual, \idemeo, \surfunceo, and order-preserving on ambient modules and on submodules. Hence setting $\alpha(M)=0_M^\li$, $\alpha$ is \funcss\ and \radss.
\end{prop}

\begin{proof}
The first line follows from \cite[Proposition 2.4]{nmeUlr-lint}, and the second from Proposition \ref{pr:ssres}.
\end{proof}

\begin{prop}
Let $R$ be a complete local Noetherian ring. Setting $\alpha(M)=0_M^\li$, $\alpha$ satisfies the hypotheses of Theorem \ref{thm:preradical}, and hence its conclusions. In particular, 
\[\alpha^\dual(R)=\ann(\alpha(E))=\ann(0_E^\li)\]
and
\[\alpha^\dual_f(R)=\ann(\alpha_f(E))=\bigcap_{I \text{ an ideal}} (I:I^-).\]
This implies that $\alpha^\dual(R)$ agrees with $\tau_{\cM}$ as defined in \cite{nmeUlr-lint}, and $\alpha^\dual_f(R)$ agrees with $\tau_{\mathcal{I}}$.
\end{prop}

\begin{proof}
This follows from Remark \ref{rem:tcicmprimary}, Theorem \ref{thm:preradical}, and Definition/Proposition 3.1 of \cite{nmeUlr-lint}.
\end{proof}

As a consequence, we know the following about $\alpha^\dual(R)$:

\begin{cor}[{\cite{nmeUlr-lint}}]
Let $(R,m)$ be a complete local Noetherian ring.
\begin{enumerate}
\item If $\dim(R)=0$, then $\alpha(M)=mM$ and $\alpha^\dual(R)=\Soc(R)$ \cite[Propositions 4.1 and 4.2]{nmeUlr-lint}.
\item If $\dim(R)=1$ and $R$ has an infinite residue field, then $R$ is Cohen-Macaulay if and only if $\alpha^\dual_f(R)$ is equal to the conductor of $R$ in its integral closure. Otherwise the conductor is equal to $R$ whereas $\alpha^\dual_f(R)$ is proper \cite[Theorem 4.4]{nmeUlr-lint}.
\item If $R$ is either excellent or the homomorphic image of a Gorenstein local ring, $\dim(R) \ge 2$, and $R$ is equidimensional with no embedded primes, then $\alpha^\dual(R)=\alpha^\dual_f(R)=0$ \cite[Theorem 4.5]{nmeUlr-lint}.
\end{enumerate}
\end{cor}

\section{Special case: almost ring theory, Heitmann closures and post-2016 mixed characteristic closures}\label{sec:almost}

Next we apply our structure to the example of a mixed characteristic closure operation defined in \cite{PeRG}. This demonstrates how our results can be applied to the almost ring theory setting, and to a closure operation that appears distinct from a module closure.

\begin{defn}
\label{def:ppowerroots}
Let $(R,m)$ be a complete local Noetherian ring of dimension $d>0$ and mixed characteristic $(0,p)$, $T$ an $R$-algebra, and $\pi \in T$ such that $T$ contains a compatible system of $p$-power roots of $\pi$, i.e. a set of elements $\{\pi^{1/p^n}\}_{n \ge 1}$ such that $\left(\pi^{1/p^n}\right)^{p^n}=\pi$ and for $m<n,$ $\left(\pi^{1/p^n}\right)^{p^m}=\pi^{1/p^{n-m}}$. We can denote this system of $p$-power roots of $\pi$ by $\pi^{1/p^\infty}$.
\end{defn}

\begin{defn}[c.f. {\cite{PeRG}}]
\label{def:mixedcharclosure}
Let $R$, $T$, and $\pi$ be as above.  We define a closure operation cl by $u \in N_M^\cl$ if for all $n>0$,
\[ \pi^{1/p^n} \otimes u \in \im(T \otimes N \to T \otimes M).
\]
\end{defn}

We describe the dual of this closure operation below.

\begin{prop}
\label{pr:mixedchardual}
Let $\alpha(M)=0_M^\cl$ with cl the closure operation from Definition \ref{def:mixedcharclosure}. Under the hypotheses of Definition \ref{def:ppowerroots},
let 
\[\beta(M)=\{z \in M : \pi^{1/p^n} \otimes z=0 \in T \otimes_R M \text{ for all } n>0\}.\]
Then 
\[\beta=\alpha^\dual.\]
\end{prop}

\begin{proof}
By Proposition \ref{pr:trbc} we have $\alpha=\sum_{n>0} \tr_{\pi^{1/p^n},T}$. Then $\rho=\varinjlim_n \tr_{\pi^{1/p^n},T}$. By Proposition \ref{pr:duallim1}, 
\[\rho^\dual=\varprojlim_n \left(\tr_{\pi^{1/p^n},T}\right)^\dual.\]
So it suffices to show that 
\[\left(\tr_{\pi^{1/p^n},T}\right)^\dual(M)=\{z \in M : \pi^{1/p^n} \otimes z=0 \in T \otimes_R M\}.\]
By Theorem \ref{thm:tracedual}, we have
\[\left(\tr_{\pi^{1/p^n},T}\right)^\dual=\tom_{\pi^{1/p^n},T}.\]
By Proposition \ref{pr:tombc}, this gives us the desired result.
\end{proof}

\begin{rem}
With cl as above, $\ann_R (0_E^\cl)=\varinjlim_n \tr_{\pi^{1/p^n},T}(R)$ by Theorem \ref{thm:preradical}.
\end{rem}

In \cite{PeRG}, P\'{e}rez and the second named author studied this closure operation in the context of almost balanced big \CM\ algebras.

\begin{defn}[{\cite{An-DS}}]
$T$ is an \textit{almost (balanced) big \CM\ algebra with respect to $\pi^{1/{p^\infty}}$} if $T/mT$ is not almost 0 with respect to $\pi^{1/{p^\infty}}$ (i.e., it is not the case that $\pi^{1/p^n}T/mT=0$ for all $n \gg 0$), and for every \sop\ $x_1,\ldots,x_d$ on $R$,
\[ \pi^{1/p^n} \cdot \frac{(x_1,\ldots,x_i):_T x_{i+1}}{(x_1,\ldots,x_i)}=0
\]
for all $n>0$, $0 \le i \le d-1$.
\end{defn}

The following result now follows immediately from Proposition \ref{pr:mixedchardual}.

\begin{thm}[{\cite[Theorem 6.5]{PeRG}}]
\label{perezrg}
Let $R$ be a complete local domain, cl be as above, and additionally assume that $T$ is $p$-torsion free, $\pi$ is a \nzd, and $T$ is an almost balanced big \CM\ algebra. Then
\[\ann_R (0_E^\cl)=\sum_{n>0} \sum_{\psi:T \to R} \psi(\pi^{1/p^n}T)=\sum_{n>0} \im(T^* \otimes \pi^{1/p^n}T \to R),\]
where the map sends $\psi \otimes \pi^{1/p^n}t \mapsto \psi(\pi^{1/p^n} t)$.
\end{thm}

\section{Good localization is dual to good colocalization}\label{sec:coloc}
 
In this section, $R$ is a commutative Noetherian ring with identity.

We introduce the following concept, with notation and terminology due to Richardson:

\begin{defn}[{\cite{Ri-cosupp}}]
Let $(R,\m)$ be a Noetherian local ring, $\p \in \Spec R$, and $M$ an $R$-module.  Then the \emph{(Richardson) colocalization} of $M$ with respect to $\p$ is the $\widehat{R_\p}$-module \[
{}^\p M := \Hom_R(\Hom_R(M, E_R(R/\m)), E_R(R/\p)).
\]
\end{defn}

This concept had earlier been explored (though without an explicit name, and with different notation (namely ${M^\vee}^{\vee_\p}$) by K. Smith \cite{Sm-param, Smtest}.  We start by recalling the following useful principle.

\begin{lemma}[{\cite[Lemma 3.1(iv)]{Sm-param}}]
\label{smithlemma}
Let $(R,\m)$ be a Noetherian local ring, let $E=E_R(R/\m)$, and let $M$ be an $R$-submodule of $E$.  Then \[
M^\vee \cong \hat{R}/\ann_{\hat{R}}(M).
\]
\end{lemma}

Note that when $M$ is an artinian $R$-module, ${}^\p M$ is an artinian $\widehat{R_\p}$-module \cite[Theorem 2.3]{Ri-cosupp}.  Also note that ${}^\p (-)$ is an exact covariant functor on the category of $R$-modules, as it consists of applying $\Hom$ into an injective $R$-module twice.  We have the following:

\begin{thm}\label{thm:colocss}
Let $(R,\m)$ be a complete Noetherian local ring, and let $\p \in \Spec R$.  Let $\sigma$ be a submodule selector on Artinian $R$-modules and $\sigma'$ a submodule selector on artinian $\widehat{R_\p}$-modules.  Let $\tau = \sigma^\dual$ (over $R$) and $\tau' = (\sigma')^\dual$ (over $\widehat{R_\p}$). Then $\tau'(\widehat{N_\p}) = \widehat{\tau(N)_\p}$ for all finitely generated $R$-modules $N$ if and only if $\sigma'({}^\p M) = {}^\p \sigma(M)$ for all Artinian $R$-modules $M$.
\end{thm}

\begin{proof}
Let $^{\vee'}$ denote the Matlis dual operation over the ring $\widehat{R_\p}$.  That is, if $U$ is an $\widehat{R_\p}$-module, then $U^{\vee'} := \Hom_{\widehat{R_\p}}(U, E_R(R/\p))$.

First suppose $\tau'(\widehat{N_\p}) = \widehat{\tau(N)_\p}$ for all finitely generated $R$-modules $N$.  Let $M$ be an Artinian $R$-module.  Then \begin{align*}
{}^\p \sigma(M) &= (\sigma(M)^\vee \otimes_R \widehat{R_\p})^{\vee'} = (\tau^\dual(M)^\vee \otimes_R \widehat{R_\p})^{\vee'} \\
&= \left(\frac{M^\vee}{\tau(M^\vee)} \otimes_R \widehat{R_\p}\right)^{\vee'} \cong \left( \frac{\widehat{(M^\vee)_\p}}{\widehat{\tau(M^\vee)_\p}}\right)^{\vee'} \\
&= (\tau')^\dual\left(\widehat{(M^\vee)_\p}^{\vee'}\right) = \sigma'({}^\p M).
\end{align*}

Conversely, suppose $\sigma'({}^\p M) = {}^\p \sigma(M)$ for all Artinian $R$-modules $M$.  Let $N$ be a finitely generated $R$-module.  Then \begin{align*}
\tau'(\widehat{N_\p}) &= (\sigma')^\dual(\widehat{N_\p}) = \left( \frac{(\widehat{N_\p})^{\vee'}}{\sigma'((\widehat{N_\p})^{\vee'})}\right)^{\vee'} = \left(\frac{{}^\p (N^\vee)} {\sigma'({}^\p (N^\vee))}\right)^{\vee'} \\
&= \left(\frac{{}^\p(N^\vee)} {{}^\p \sigma(N^\vee)}\right)^{\vee'} \cong \left(^\p \left(\frac{N^\vee} {\sigma(N^\vee)}\right)\right)^{\vee'} = ({}^\p (\sigma^\dual(N)^\vee))^{\vee'} \\
&= \widehat{\sigma^\dual(N)_\p} = \widehat{\tau(N)_\p} \qedhere
\end{align*}
\end{proof}

In particular, we have:

\begin{thm}\label{thm:tccoloc}
Let $R$ be a complete Noetherian reduced $F$-finite ring of prime characteristic $p>0$.  Let $M$ be an artinian $R$-module.  Then for any prime ideal $\p$, we have ${}^\p(0^*_M) = 0^*_{{}^\p M}$, i.e., tight closure co-localizes.
\end{thm}

\begin{proof}
By Proposition \ref{pr:tint} (2), tight interior agrees with the dual of tight closure in the sense of this paper.
Recall that in this circumstance, tight interior localizes and commutes with completion \cite[Corollary 2.11]{nmeSc-tint}. That is, for any finitely generated $R$-module $N$, we $\tau'(\widehat{N_\p}) = \widehat{\tau(N)_\p}$, where $\tau$ is tight interior on $R$-modules and $\tau'$ is tight interior on $\widehat{R_\p}$-modules. Then setting $\sigma=\tau^\dual$ to be the submodule selector $\sigma(M)=0_M^*$, we apply Theorem \ref{thm:colocss} to get the result.
\end{proof}

As a special case, we recover the following fact, implicit in \cite{LySm-Test} as evidenced by the second proof below.

\begin{cor}\label{cor:tcccp}
Let $R$ be a complete Noetherian reduced $F$-finite ring of prime characteristic $p>0$.  Let $E$ be the injective hull of the residue field of $R$.  Let $\p \in \Spec R$.  Then \[
0^*_{E_R(R/\p)} = {}^\p (0^*_E).
\]
\end{cor}

\begin{proof}[Proof 1, via new methods]
By definition of ${}^\p(-)$ and since $E^\vee = R$, we have ${}^\p E = E_R(R/\p)$.  The result then follows from Theorem~\ref{thm:tccoloc} with $E=M$.
\end{proof}

\begin{proof}[Proof 2, via old methods]
We have the following chain of equalities: \begin{align*}
0^*_{E_R(R/\p)} &= \Hom_{\widehat{R_\p}}(\widehat{R_\p} / \ann_{\widehat{R_\p}} (0^*_{E_R(R/\p)}), E_R(R/\p)) &\text{by Lemma~\ref{smithlemma}} \\
&=\Hom_{\widehat{R_\p}}(\widehat{R_\p} / \ann_R(0^*_E)\widehat{R_\p}, E_R(R/\p)) &\text{by {\cite[Theorem 7.1]{LySm-Test}}}\\
&=\Hom_{\widehat{R_\p}}(\widehat{R_\p} \otimes_R R/\ann_R(0^*_E), E_R(R/\p)) &\text{by flatness of $\widehat{R_\p}$ over $R$}\\
&\cong \Hom_R(R/\ann_R(0^*_E), \Hom_{\widehat{R_\p}}(\widehat{R_\p}, E_R(R/\p))) &\text{by $\Hom$-$\otimes$ adjunction}\\
&\cong \Hom_R(R/\ann_R(0^*_E), E_R(R/\p)) \\
&= \Hom_R((0^*_E)^\vee, E_R(R/\p)) &\text{by Lemma~\ref{smithlemma}} \\
&= {}^\p(0^*_E) \qedhere
\end{align*}
\end{proof}

We also get the following as a corollary:
\begin{thm}
Let $R$ be a complete Noetherian local ring and $B$ a finitely generated $R$-module.  Let $M$ be an artinian $R$-module.  Then \[
 {}^\p(0^{\cl_B}_M) = 0^{\cl_{\widehat{B_\p}}}_{{}^\p M}.
\]
In particular, we have \[
{}^\p (0^{\cl_B}_E) = 0^{\cl_{\widehat{B_\p}}}_{E_R(R/\p)}.
\]
\end{thm}

\begin{proof}
This follows from Theorems~\ref{thm:colocss}, \ref{thm:traceloc}, and \ref{thm:tracedual}.
\end{proof}

Finally, we observe a curious relationship between $I$-adic completion, $\p$-colocalization, and $\p$-adic completion:

\begin{thm}
Let $R$ be a complete Noetherian local ring, let $\p \in \Spec R$, and $M$ an artinian $R$-module.  Then \[
\bigcap_{n\geq 0}\left( (\widehat{I R_\p})^n \cdot {}^\p M\right) = {}^\p \bigg(\bigcap_{n\geq 0} (I^n M)\bigg).
\]
That is, the colocalization at $\p$ of the kernel of the $I$-adic completion map on $M$ is the same as the kernel of the $\widehat{IR_\p}$-adic completion map on the colocalization of $M$ at $\p$.

Specializing to the case $M=E$, we have \[
\bigcap_{n\geq 0}\left((\widehat{I R_\p})^n \cdot E_R(R/\p)\right) = {}^\p \bigg(\bigcap_{n\geq 0} (I^n E)\bigg).
\]
\end{thm}

\begin{proof}
Let $\sigma$ be the submodule selector on $R$-modules given by $M \mapsto \bigcap_n I^n M$.  Then by Corollary~\ref{cor:H0Idual}, $\tau = \sigma^\dual = H^0_I$.  Similarly, letting $\sigma'$ be the submodule selector on $\widehat{R_\p}$-modules given by $M \mapsto \bigcap_n (\widehat{IR_\p})^n M$, we have $\tau' = (\sigma')^\dual = H^0_{\widehat{IR_\p}}$. Since local cohomology commutes with flat base change \cite[4.3.2 Flat Base Change Theorem]{BrSh-locobook}, and since completion of a finitely generated module coincides with tensoring with the completed ring, we have \begin{align*}
\tau'(\widehat{N_\p})&=  H^0_{\widehat{IR_\p}} (\widehat{N_\p}) = H^0_{I \widehat{R_\p}}(N \otimes_R \widehat{R_\p}) \\
&= H^0_I(N) \otimes_R \widehat{R_\p} = \widehat{H^0_I(N)_\p}= \widehat{\tau(N)_\p}
\end{align*} for all finitely generated $R$-modules $N$.  Thus, by Theorem~\ref{thm:colocss}, we have for all Artinian $R$-modules $M$ that \[
\bigcap_{n\geq 0}\left((\widehat{I R_\p})^n \cdot {}^\p M\right) = \sigma'(^\p M) = {}^\p \sigma(M) = {}^\p \bigg(\bigcap_{n\geq 0} (I^n M)\bigg).
\]
The final statement follows from the fact that ${}^\p E = E_R(R/\p)$.
\end{proof}

\section*{Acknowledgment}
We are grateful for suggestions from Geoffrey Dietz, Jesse Elliott, Haydee Lindo, Karen Smith, Janet Vassilev, and the anonymous referee, all of which improved the paper.  

\providecommand{\bysame}{\leavevmode\hbox to3em{\hrulefill}\thinspace}
\providecommand{\MR}{\relax\ifhmode\unskip\space\fi MR }
\providecommand{\MRhref}[2]{%
  \href{http://www.ams.org/mathscinet-getitem?mr=#1}{#2}
}
\providecommand{\href}[2]{#2}


\begin{thebibliography}{CLVW06}

\bibitem[And18]{An-DS}
Yves Andr\'{e}, \emph{La conjecture du facteur direct}, Publ. Math. Inst.
  Hautes \'{E}tudes Sci. \textbf{127} (2018), 71--93. \MR{3814651}

\bibitem[BM10]{BM-unloc}
Holger Brenner and Paul Monsky, \emph{Tight closure does not commute with
  localization}, Ann. of Math. (2) \textbf{171} (2010), no.~1, 571--588.

\bibitem[BS98]{BrSh-locobook}
Markus~P. Brodmann and Rodney~Y. Sharp, \emph{Local cohomology: an algebraic
  introduction with geometric applications}, Cambridge Studies in Advanced
  Mathematics, vol.~60, Cambridge Univ. Press, Cambridge, 1998.

\bibitem[CLVW06]{LiftingModules}
John Clark, Christian Lomp, Narayanaswami Vanaja, and Robert Wisbauer,
  \emph{Lifting modules}, Frontiers in Mathematics, Birkh\"{a}user Verlag,
  Basel, 2006.

\bibitem[Die10]{Di-clCM}
Geoffrey~D. Dietz, \emph{A characterization of closure operations that induce
  big {C}ohen-{M}acaulay modules}, Proc. Amer. Math. Soc. \textbf{138} (2010),
  no.~11, 3849--3862.

\bibitem[DT95]{DikTho-closure}
Dikran Dikranjan and Walter Tholen, \emph{Categorical structure of closure
  operators}, Mathematics and its Applications, vol. 346, Kluwer, Dordrecht,
  1995.

\bibitem[EG89]{EvGr-Ord}
E.~Graham Evans, Jr. and Phillip~A. Griffith, \emph{Order ideals}, Commutative
  algebra ({B}erkeley, {CA}, 1987), Math. Sci. Res. Inst. Publ., vol.~15,
  Springer, New York, 1989, pp.~213--225. \MR{1015519}

\bibitem[EJ00]{EnJe-relhombook}
Edgar~E. Enochs and Overtoun M.~G. Jenda, \emph{Relative homological algebra},
  De Gruyter Expositions in Mathematics, vol.~30, Walter de Gruyter \& Co.,
  Berlin, 2000.

\bibitem[Eps12]{nme-guide2}
Neil Epstein, \emph{A guide to closure operations in commutative algebra},
  Progress in Commutative Algebra 2 (Berlin/Boston) (Christopher Francisco, Lee
  Klingler, Sean Sather-Wagstaff, and Janet~C. Vassilev, eds.), De Gruyter
  Proceedings in Mathematics, De Gruyter, 2012, pp.~1--37.

\bibitem[ES14]{nmeSc-tint}
Neil Epstein and Karl Schwede, \emph{A dual to tight closure theory}, Nagoya
  Math. J. \textbf{213} (2014), 41--75.

\bibitem[EU]{nmeUlr-lint}
Neil Epstein and Bernd Ulrich, \emph{Liftable integral closure},
  arXiv:1309.6966 [math.AC], to appear in J. Commut. Algebra.

\bibitem[EY12]{nmeYao-flat}
Neil Epstein and Yongwei Yao, \emph{Criteria for flatness and injectivity},
  Math. Z. \textbf{271} (2012), no.~3-4, 1193--1210.

\bibitem[HH90]{HHmain}
Melvin Hochster and Craig Huneke, \emph{Tight closure, invariant theory, and
  the {Brian\c{c}on}-{Skoda} theorem}, J. Amer. Math. Soc. \textbf{3} (1990),
  no.~1, 31--116.

\bibitem[Hoc07]{HoFNDTC}
Melvin Hochster, \emph{Foundations of {T}ight {C}losure {T}heory},
  \url{http://www.math.lsa.umich.edu/~hochster/711F07/fndtc.pdf}, 2007.

\bibitem[HS06]{HuSw-book}
Craig Huneke and Irena Swanson, \emph{Integral closure of ideals, rings, and
  modules}, London Math. Soc. Lecture Note Ser., vol. 336, Cambridge Univ.
  Press, Cambridge, 2006.

\bibitem[HT04]{HaTa-gentest}
Nobuo Hara and Shunsuke Takagi, \emph{On a generalization of test ideals},
  Nagoya Math. J. \textbf{175} (2004), 59--74.

\bibitem[Hun96]{HuTC}
Craig Huneke, \emph{Tight closure and its applications}, CBMS Reg. Conf. Ser.
  in Math., vol.~88, Amer. Math. Soc., Providence, RI, 1996.

\bibitem[Lam99]{Lam99}
T.~Y. Lam, \emph{Lectures on modules and rings}, Graduate Texts in Mathematics,
  vol. 189, Springer-Verlag, New York, 1999. \MR{1653294}

\bibitem[Lin17]{Lin17}
Haydee Lindo, \emph{Trace ideals and centers of endomorphism rings of modules
  over commutative rings}, J. Algebra \textbf{482} (2017), 102--130.
  \MR{3646286}

\bibitem[LS01]{LySm-Test}
Gennady Lyubeznik and Karen~E. Smith, \emph{On the commutation of the test
  ideal with localization and completion}, Trans. Amer. Math. Soc. \textbf{353}
  (2001), no.~8, 3149--3180. \MR{1828602}

\bibitem[PR20]{PeRG}
Felipe P\'{e}rez and Rebecca R.G., \emph{Characteristic-free test ideals},
  {arXiv}:1907.02150, to appear in Trans. Amer. Math. Soc., 2020.

\bibitem[R.G16]{RG-bCMsing}
Rebecca R.G., \emph{Closure operations that induce big {C}ohen-{M}acaulay
  modules and classification of singularities}, J. Algebra \textbf{467} (2016),
  237--267.

\bibitem[Ric06]{Ri-cosupp}
Andrew~S. Richardson, \emph{Co-localization, co-support and local homology},
  Rocky Mountain J. Math. \textbf{36} (2006), no.~5, 1679--1703.

\bibitem[Smi94]{Sm-param}
Karen~E. Smith, \emph{Tight closure of parameter ideals}, Invent. Math.
  \textbf{115} (1994), no.~1, 41--60.

\bibitem[Smi95]{Smtest}
\bysame, \emph{Test ideals in local rings}, Trans. Amer. Math. Soc.
  \textbf{347} (1995), no.~9, 3453--3472.

\end{thebibliography}
\end{document}